\documentclass[11pt]{article}

\usepackage[margin=1.12in]{geometry}
\usepackage{amsmath,amssymb,amsthm,mathtools}
\usepackage[round,authoryear]{natbib}
\usepackage[colorlinks,citecolor=blue,linkcolor=blue,urlcolor=blue]{hyperref}
\usepackage{enumitem}
\usepackage{booktabs}
\usepackage{microtype}
\usepackage[utf8]{inputenc}
\usepackage[T1]{fontenc}
\usepackage[osf,sc]{mathpazo}

\newtheorem{theorem}{Theorem}[section]
\newtheorem{proposition}[theorem]{Proposition}
\newtheorem{lemma}[theorem]{Lemma}
\newtheorem{corollary}[theorem]{Corollary}
\theoremstyle{definition}

\newtheorem{remark}[theorem]{Remark}

\newtheorem{assumption}[theorem]{Assumption}

\newcommand{\E}{\mathbb E}
\newcommand{\R}{\mathbb R}
\newcommand{\one}{\mathbf 1}
\newcommand{\Hd}{\{0,1\}^d}
\newcommand{\Hs}{\{0,1\}^s}
\newcommand{\cE}{\mathcal E}

\newcommand{\cM}{\mathcal M}
\newcommand{\cG}{\mathcal G}
\newcommand{\KL}{\mathrm{KL}}
\newcommand{\Var}{\mathrm{Var}}

\newcommand{\wh}{\widehat}
\newcommand{\wt}{\widetilde}
\newcommand{\norm}[1]{\left\lVert #1\right\rVert}

\title{Resolvable Fourier Degree and Minimax Estimation on the Boolean Hypercube}
\author{%
  G\'erard Biau\\[4pt]
  \small Sorbonne Universit\'e, Institut universitaire de France\\[2pt]
  \small \texttt{gerard.biau@sorbonne-universite.fr}
}

\date{}

\begin{document}
\maketitle

\begin{abstract}
\noindent We study high-dimensional nonparametric regression on the Boolean
hypercube \(\{0,1\}^d\), from a sample of size \(n\), under a Fourier--Sobolev budget, i.e., a 
bound \(B\) on the total
quadratic influence of the regression function. The central object is the
resolvable Fourier degree \(D^*(d,n)\), defined as the largest
interaction order \(D\) such that the dimension of the degree-\(D\)
Fourier--Walsh space does not exceed \(n\).  We prove that the
minimax risk over Fourier--Sobolev ellipsoids is of order
\(B/D^*(d,n)\): estimation is possible up to the statistically
visible interaction order, and the budget controls the residual mass
carried by higher-order interactions.  We then study monotone
regression under the same budget.  For monotone functions supported on a 
known set of \(s\) coordinates, with \(s\) of order \(\log n\), we obtain a matching minimax rate of
order \(B/s\). Finally, we consider an
unknown-support model in which the function is monotone and depends
on at most \(s\) coordinates; its minimax risk is governed by the
intrinsic dimension \(s\) together with the logarithmic cost of
selecting the active support.
\end{abstract}


\section{Introduction}\label{sec:intro}

\subsection*{Context and motivation}

Binary covariates arise naturally in many statistical problems.  In clinical medicine, a patient may be described by the presence or absence of risk factors such as smoking, hypertension, diabetes, family history, or genetic mutations.  In credit scoring, an applicant may be represented by binary indicators recording previous default, employment status, home ownership, and other financial flags.  In pharmacogenomics, the response to a treatment may depend on the presence or absence of deleterious genetic variants. In such examples, monotonicity is a natural structural assumption, since the occurrence of an additional risk factor should not make the outcome less likely; see, for instance, \citet{ben-david1995,feelders2010}.

The present paper was motivated by this monotone-regression setting, but the main principle that emerges is more general.  We observe
\[
  Y_j=f(X_j)+\varepsilon_j,\qquad j=1,\ldots,n,
\]
where \(X_1,\ldots,X_n\) are independent and uniformly distributed on
\(\{0,1\}^d\). The variables \(\varepsilon_1,\ldots,\varepsilon_n\) are i.i.d., centered and sub-Gaussian with variance proxy \(\sigma^2\), and are independent of the design. Throughout the paper, the index \(j\) is used for observations, whereas the index \(i\) is reserved for coordinates in \([d]\). The loss is integrated squared error with respect to the uniform measure on the cube.  

Although \(\Hd\) is finite, its cardinality grows so rapidly with \(d\) that, already for moderate values of the dimension, direct exploration of the domain is statistically impossible. For instance,
\[
2^{100}\simeq 1.27\times 10^{30},
\qquad
2^{1000}\simeq 1.07\times 10^{301}.
\]
The latter number exceeds the usual estimate, of order \(10^{80}\), for the number of atoms in the observable universe. Even a sample of size one million therefore probes only an infinitesimal fraction of the domain.

Monotonicity alone does not remove this difficulty.  The number of Boolean monotone functions is the Dedekind number $M(d)$, and
\[
\log_2 M(d)\sim \binom{d}{\lfloor d/2\rfloor}
\sim \sqrt{\frac{2}{\pi d}}\,2^d.
\]
Thus even the class of Boolean-valued monotone functions is far too large to be learned without an additional regularity assumption.  In applications, however, one often expects a form of effective low complexity.  Many variables may be present, but only some combinations of them have a substantial effect.  The problem is to express this principle in a way that is compatible with the geometry of the cube and leads to a sharp statistical theory.

The Fourier--Walsh expansion provides such a language. It decomposes a function into interaction levels of increasing order, and is a standard tool in the analysis of Boolean functions; see, for example, \citet{odonnell2014}. For \(x=(x_1,\ldots,x_d)\in\Hd\) and \(S\subseteq[d]\), define
\[
  \chi_S(x)=\prod_{i\in S}(2x_i-1),
  \qquad \chi_\emptyset\equiv 1.
\]
If \(X\) is uniformly distributed on \(\Hd\), the Fourier--Walsh coefficient of \(f\) associated with \(S\) is
\[
  \wh f(S)=\E\big[f(X)\chi_S(X)\big],
\]
where the expectation is taken with respect to the law of \(X\). With these conventions,
\[
  f(x)=\sum_{S\subseteq[d]}\wh f(S)\chi_S(x),
  \qquad x\in\Hd.
\]
The coefficient \(\wh f(S)\) measures the contribution of the interaction among the variables indexed by \(S\). Coefficients with \(|S|=1\) describe main effects, coefficients with \(|S|=2\) describe pairwise interactions, and higher-degree coefficients describe increasingly complex interactions.

For \(0\leqslant D\leqslant d\), the Fourier expansion truncated at
degree \(D\) contains
\[
  N(d,D)=\sum_{k=0}^{D} \binom{d}{k}
\]
coefficients. The quantity \(N(d,D)\) is the first statistical bottleneck. If it is much larger than \(n\), the whole degree-\(D\) Fourier projection cannot be estimated reliably as a block from \(n\) random observations.

This bottleneck is already visible in simple numerical examples. If \(d=1000\) and \(n=10^6\), then
\[
  N(1000,2)=1+1000+\binom{1000}{2}=500\,501.
\]
Thus the number of Fourier--Walsh interactions of degrees \(0\), \(1\), and
\(2\) remains within the sample-size budget. The next layer, however, contains
\[
  \binom{1000}{3}=166\,167\,000
\]
additional coefficients, far more than the available number of observations. In this sense, degree \(2\) is resolvable at this sample size, whereas degree \(3\) is not. If instead \(d=100\) and \(n=10^6\), the corresponding boundary moves to degree \(3\), since
\[
  N(100,3)=166\,751,
  \qquad
  N(100,4)=4\,087\,976.
\]
These examples already contain the main phenomenon of the paper. The sample size determines the highest interaction order up to which the cube is statistically visible.

\subsection*{The resolution principle}

The preceding counting argument suggests a simple notion of statistical
resolution. We define the resolvable Fourier degree by
\begin{equation}\label{eq:Dstar-intro}
  D^*(d,n)=\max\{0\leqslant D\leqslant d:N(d,D)\leqslant n\}.
\end{equation}
Thus \(D^*(d,n)\) is the largest degree up to which all Fourier--Walsh
interactions can be retained while keeping the dimension of the linear
space \(\operatorname{span}\{\chi_S: |S|\leqslant D\}\), namely
\(N(d,D)\), within the sample-size budget.

If \(n\geqslant 2^d\), then \(D^*(d,n)=d\).  This is the saturated
regime: the whole Fourier expansion contains exactly \(2^d\) coefficients,
so the full Fourier--Walsh basis is, at least dimensionally, within the
sample-size budget.  In this case, there is no unresolved Fourier layer.
The main focus of the paper is the opposite, non-saturated regime, where 
\(D^*(d,n)<d\) and the first unresolved Fourier layer has degree
\(D^*(d,n)+1\).

Our main result, Theorem~\ref{thm:ellipsoid}, shows that this dimensional
notion of resolution is not merely a heuristic. Once the Fourier coefficients are controlled through the Fourier--Sobolev
energy
\[
  I^{(2)}(f)
  =
  4\sum_{S\ne\emptyset}|S|\wh f(S)^2
  \leqslant B,
\]
the resolvable Fourier degree determines the minimax risk.  The quantity \(I^{(2)}(f)\) is the Fourier--Walsh form of a discrete
first-order Sobolev energy.  Equivalently, \(I^{(2)}(f)\) is the squared
\(L^2\)-norm of the discrete gradient.  In Fourier--Walsh coordinates, the
factor \(|S|\) penalizes each coefficient according to the degree of the
corresponding interaction.  This immediately gives the tail estimate
\[
  \sum_{|S|>D}\wh f(S)^2
  \leqslant
  \frac{1}{D+1}\sum_{|S|>D}|S|\wh f(S)^2
  \leqslant
  \frac{1}{D+1}\sum_{S\ne\emptyset}|S|\wh f(S)^2
  \leqslant
  \frac{B}{4(D+1)}.
\]
Thus, retaining all Fourier--Walsh coefficients up to degree \(D\) leaves a
squared truncation error at most \(B/(4(D+1))\).

This tail bound suggests a natural estimator.  We truncate the
Fourier--Walsh expansion at degree \(D\) and estimate only the coefficients
with \(|S|\leqslant D\).  For each such \(S\), define
\begin{equation}\label{eq:estimator-intro}
  \wt f(S)=\frac{1}{n}\sum_{j=1}^n Y_j\chi_S(X_j),
\end{equation}
and form the empirical projection
\[
  \wt f_D(x)=\sum_{|S|\leqslant D}\wt f(S)\chi_S(x).
\]
The stochastic error is governed by the number \(N(d,D)\) of retained
Fourier--Walsh coefficients, while the squared truncation error is controlled
by \(B/(D+1)\).  Thus the same cutoff degree \(D\) has two complementary
roles.  Through \(N(d,D)\), it determines the dimension of the retained
linear space.  Through the Fourier--Sobolev tail bound, it determines the
deterministic truncation error.  For fixed \(B\) and fixed noise level, the optimal cutoff lies just below
the resolution boundary, and the resulting risk is of order
\[
  \frac{B}{D^*(d,n)}.
\]
Theorem~\ref{thm:ellipsoid} proves that this rate is minimax optimal over
the centered Fourier--Sobolev ellipsoid.

\subsection*{Finiteness versus resolvable Fourier complexity}

The finiteness of \(\{0,1\}^d\) should not be confused with a parametric
simplification.  Although every regression function can be viewed as a vector
in \(\mathbb R^{2^d}\), this representation is statistically unhelpful in the
regimes considered here, where \(2^d\) is far larger than the sample size.
The relevant question is not whether the ambient function space is finite
dimensional, but how many structured Fourier--Walsh directions can be
estimated from the data.

In our setting, these directions are ordered by interaction degree.  The
effective statistical complexity is therefore measured by the resolvable
Fourier degree \(D^*(d,n)\), rather than by the total number \(2^d\) of
Fourier--Walsh coefficients.  Hence the rate \(B/D^*(d,n)\) captures a
genuine high-dimensional effect: as the cube dimension \(d\) grows, the data
reach progressively fewer levels of the Fourier hierarchy.  This dependence
would be invisible from the crude finite-dimensional viewpoint based only on
the size \(2^d\) of the domain.

\subsection*{Monotonicity and intrinsic dimension}

The Fourier--Sobolev result is not specific to monotone functions.  It gives
matching upper and lower minimax bounds for all functions whose
Fourier--Walsh coefficients satisfy a quadratic smoothness constraint.
Monotone regression raises a different question.  If a function is known to
be nondecreasing in each coordinate, can this order structure improve the
statistical problem, and how should the dimension of such a model be
measured?

This connection is most naturally expressed through discrete derivatives.
For \(x=(x_1,\ldots,x_d)\in\{0,1\}^d\) and \(i\in\{1,\ldots,d\}\), let
\(x^{i\to a}\) denote the point obtained from \(x\) by replacing its
\(i\)-th coordinate by \(a\in\{0,1\}\).  Define
\[
  \Delta_i f(x)=f(x^{i\to1})-f(x^{i\to0}).
\]
Monotonicity is the sign condition \(\Delta_i f(x)\geqslant0\) for every
\(i\) and \(x\).  On the other hand, the Fourier--Sobolev energy admits the
derivative representation
\[
  I^{(2)}(f)=\sum_{i=1}^d \E\big[(\Delta_i f(X))^2\big],
\]
where \(X\) is uniformly distributed on \(\Hd\).  Thus \(I^{(2)}(f)\)
measures the squared size of the discrete gradient, while monotonicity fixes
its sign.  This combination is central to the monotone constructions
developed later in the paper.  The functions are built from nonnegative
discrete increments, while the Fourier--Sobolev constraint bounds their
total quadratic size.

Monotonicity alone, however, does not remove the high-dimensional nature of
the problem.  A monotone function may still depend on all \(d\) coordinates.
We therefore distinguish the ambient dimension \(d\) from the intrinsic
dimension, defined as the number of active coordinates.  The intrinsic
monotone model considered in this paper assumes that at most \(s\)
coordinates are active, but does not assume that this active set is known.
The resulting minimax rate separates the Fourier resolution on the active
subcube from the combinatorial cost of identifying this subcube among the
\(d\) ambient coordinates.

\subsection*{Contributions and relation to the literature}

The main contribution of the paper is to identify \(D^*(d,n)\) as the
effective statistical resolution of the Boolean cube under a
Fourier--Sobolev constraint.  Theorem~\ref{thm:ellipsoid} proves the minimax
rate \(B/D^*(d,n)\) over the centered Fourier--Sobolev ellipsoid
\(I^{(2)}(f)\leqslant B\).  The matching upper bound is attained by the
explicit empirical Fourier projection estimator described above.

We then develop a monotone theory under the same quadratic budget.
Theorem~\ref{thm:known-support} treats the oracle case where the \(s\)
active coordinates are known and \(s\asymp\log n\).  In this setting, the
minimax rate is \(B/s\).  Theorem~\ref{thm:intrinsic-monotone} treats the
corresponding model in which the active coordinates are unknown.  The upper
bound contains the additional cost \(\log\binom ds/n\), which is the price
of support selection.  Consequently, the minimax rate remains of order
\(B/s\) whenever this selection cost is no larger than the intrinsic
estimation error.

The paper is connected to the general minimax theory based on local metric
entropy.  Results in the spirit of \citet{yang1999}, \citet{neykov2023},
and \citet{prasadan2025reg,prasadan2025} characterize minimax rates through
local packing or entropy quantities and provide procedures adapted to these
characterizations.  Related geometric viewpoints on constrained least
squares, such as \citet{chatterjee2014}, express risks in terms of Gaussian
processes or local geometry of the constraint set. Our contribution is complementary and more explicit.  In the Boolean Fourier--Sobolev geometry 
considered here, ordering the
Fourier--Walsh basis by interaction degree turns the local complexity into
the concrete resolution threshold \(D^*(d,n)\). This yields both a closed-form minimax rate and a simple
empirical Fourier projection estimator.

The paper also uses standard tools from the analysis of Boolean functions.
For background, we refer to \citet{odonnell2014}.  Relations between
discrete derivatives and Fourier--Walsh coefficients are central in this
area, from the KKL theorem \citep{kahn1988} to the
Benjamini--Kalai--Schramm bound \citep{bks1999} and later refinements.
For our purposes, the key consequence is the derivative representation of
the Fourier--Sobolev energy.  It
links the quadratic budget to discrete derivatives and is used in the
monotone lower-bound constructions.

The rest of the article is organized as follows.  Section~\ref{sec:model}
introduces the model, the Fourier--Walsh notation, and the resolvable
Fourier degree.  Section~\ref{sec:estimator} defines the empirical Fourier
projection estimator and records its risk bound.
Section~\ref{sec:main-results} states the minimax theorems.
Section~\ref{sec:discussion-related} discusses monotonicity, intrinsic
dimension, and related work.  The main proofs are given in the Appendix; the remaining proofs are collected in the Supplementary Material.

\section{Statistical model and Fourier notation}
\label{sec:model}

\subsection*{Model, risk, and Fourier basis}

We now briefly recall the statistical model and fix the notation used in
the rest of the paper.  The coordinate index is denoted by $i\in[d]$,
while the observation index is denoted by $j\in[n]$.  Subsets of
coordinates are denoted by $S,T\subseteq[d]$, and degrees by $k$, $m$, or
$D$.

Let $n\geqslant1$, and write $\mu$ for the uniform distribution on $\Hd$.  We observe
\begin{equation}\label{eq:model}
  Y_j=f(X_j)+\varepsilon_j,\qquad j=1,\ldots,n,
\end{equation}
where \(X_1,\ldots,X_n\) are independent with common distribution \(\mu\). 
The noise variables \(\varepsilon_1,\ldots,\varepsilon_n\) are i.i.d.,
centered, and independent of the design.  They are assumed to be
sub-Gaussian with parameter \(\sigma^2>0\), in the sense that
\[
  \E e^{t\varepsilon_j}
  \leqslant
  \exp(\sigma^2t^2/2),
  \qquad t\in\R,\quad j=1,\ldots,n.
\]

Throughout the paper, the risk of an estimator \(\wh f\) is measured in
\(L^2(\mu)\):
\[
  R(\wh f,f)=\E\norm{\wh f-f}_2^2,
  \qquad
  \norm{g}_2^2=\int_{\Hd} g^2\,d\mu
  =2^{-d}\sum_{x\in\Hd} g(x)^2.
\]
The expectation in the risk is taken over both the random design and the
noise; the risk therefore depends on the law of the noise, which we
denote by \(P_\varepsilon\), writing \(R_{P_\varepsilon}(\wh f,f)\)
when this dependence matters.  Let \(\mathcal P_{\sigma^2}\) be the
class of admissible noise laws, namely the distributions of centered
random variables, independent of the design, satisfying the sub-Gaussian
condition above.  For a class \(\mathcal A\), we define the minimax
risk
\[
  R_n^*(\mathcal A)
  =
  \inf_{\wh f}\,
  \sup_{f\in\mathcal A}\,
  \sup_{P_\varepsilon\in\mathcal P_{\sigma^2}}
  R_{P_\varepsilon}(\wh f,f),
\]
where the infimum is over all estimators \(\wh f\) based on the sample
\((X_j,Y_j)_{j=1}^n\).  The supremum over the noise law makes the
two-sided bounds below unambiguous: all upper bounds hold uniformly over
\(\mathcal P_{\sigma^2}\), while all lower bounds are established for
the Gaussian law \(N(0,\sigma^2)\), which belongs to
\(\mathcal P_{\sigma^2}\) and therefore bounds the supremum from
below.

The Walsh functions \(\{\chi_S:S\subseteq[d]\}\) form the Fourier--Walsh
orthonormal basis of \(L^2(\Hd,\mu)\).  Every function \(f:\Hd\to\R\)
admits the expansion
\[
  f=\sum_{S\subseteq[d]} \wh f(S)\chi_S,
  \qquad
  \norm{f}_2^2=\sum_{S\subseteq[d]} \wh f(S)^2.
\]
For an integer \(0\leqslant D\leqslant d\), we write
\[
  N(d,D)=\bigl|\{S\subseteq[d]:|S|\leqslant D\}\bigr|=\sum_{k=0}^{D} \binom{d}{k}.
\]
The resolvable Fourier degree \(D^*(d,n)\) from
\eqref{eq:Dstar-intro} is the largest degree for which the full low-degree
Fourier--Walsh space has dimension at most \(n\).  Thus \(D^*(d,n)=d\) in
the saturated case \(n\geqslant 2^d\).  Otherwise \(D^*(d,n)<d\), and adding
the Fourier layer of degree \(D^*(d,n)+1\) is the first step that makes the
full retained space exceed the sample-size budget.

\subsection*{Discrete derivatives and the Fourier--Sobolev identity}

For \(x\in\Hd\), \(i\in[d]\), and \(a\in\{0,1\}\), let \(x^{i\to a}\)
denote the point obtained from \(x\) by replacing its \(i\)-th coordinate by
\(a\) and leaving all other coordinates unchanged.  We write
\[
  \Delta_i f(x)=f(x^{i\to 1})-f(x^{i\to 0}).
\]
The coordinate contribution of the \(i\)-th discrete derivative is
\[
  I_i^{(2)}(f)=\int_{\Hd}(\Delta_i f)^2\,d\mu,
\]
and the total Fourier--Sobolev energy is
\begin{equation}\label{eq:fourier-sobolev-derivative}
  I^{(2)}(f)=\sum_{i=1}^d I_i^{(2)}(f)
  =
  \sum_{i=1}^d \int_{\Hd}(\Delta_i f)^2\,d\mu .
\end{equation}
The Fourier--Walsh expansion gives an explicit form for these increments.
If
\[
  f=\sum_{S\subseteq[d]}\wh f(S)\chi_S,
\]
then
\begin{equation}\label{eq:deriv-fourier}
  \Delta_i f
  =
  2\sum_{S\ni i}\wh f(S)\chi_{S\setminus\{i\}}.
\end{equation}
Indeed, if \(i\notin S\), then the character \(\chi_S\) does not depend on
\(x_i\), so
\[
  \chi_S(x^{i\to 1})=\chi_S(x^{i\to 0}).
\]
If \(i\in S\), writing
\[
  \chi_S=\chi_{\{i\}}\chi_{S\setminus\{i\}},
\]
and using
\[
  \chi_{\{i\}}(x^{i\to 1})-\chi_{\{i\}}(x^{i\to 0})=1-(-1)=2,
\]
gives \eqref{eq:deriv-fourier}.  Squaring \eqref{eq:deriv-fourier} and
integrating with respect to \(\mu\), the orthonormality of the Walsh
functions gives
\begin{equation}\label{eq:inf2-fourier}
  I_i^{(2)}(f)=4\sum_{S\ni i}\wh f(S)^2.
\end{equation}
Summing \eqref{eq:inf2-fourier} over \(i\) gives the Fourier--Sobolev
identity
\[
  I^{(2)}(f)=4\sum_{S\ne\emptyset}|S|\wh f(S)^2.
\]
In particular, for \(0\leqslant D\leqslant d\), the Fourier mass above
degree \(D\) satisfies
\begin{equation}\label{eq:tail-bound}
  \sum_{|S|>D}\wh f(S)^2
  \leqslant \frac{I^{(2)}(f)}{4(D+1)}.
\end{equation}
This is the regularity estimate used to control the squared truncation error
of the projection estimator.

The main minimax theorem is stated for the centered Fourier--Sobolev
ellipsoid
\begin{equation}\label{eq:ellipsoid}
  \cE_B
  =
  \left\{
    f:\Hd\to\R:\;
    \wh f(\emptyset)=0,\ 
    I^{(2)}(f)\leqslant B
  \right\}.
\end{equation}
Here \(\wh f(\emptyset)\) is the coefficient of the constant Walsh
function \(\chi_\emptyset\equiv1\); thus \(\wh f(\emptyset)=0\)
means that \(f\) has mean zero under \(\mu\).  This removes the constant
direction, which is not penalized by \(I^{(2)}\), and ensures that the class
is bounded in \(L^2(\mu)\):
\[
  \norm{f}_2^2
  =
  \sum_{S\ne\emptyset}\wh f(S)^2
  \leqslant
  \frac{I^{(2)}(f)}{4}
  \leqslant \frac{B}{4}.
\]

\section{The empirical Fourier projection estimator}\label{sec:estimator}

For a truncation degree \(D\), the empirical coefficients and the
corresponding projection are defined by \eqref{eq:estimator-intro}.  When
the target function is known to take values in \([0,1]\), as in the
monotone model, we use the clipped estimator
\[
  \wh f_D(x)=\Pi_{[0,1]}\wt f_D(x),
  \qquad
  \Pi_{[0,1]}(t)=\max(0,\min(t,1)).
\]
On the centered ellipsoid \(\cE_B\), no clipping is used.

The estimator is deliberately simple.  It keeps the empirical
Fourier--Walsh coefficients up to degree \(D\) and discards the higher
degrees.  Its risk is therefore governed by the usual bias--variance
tradeoff: the variance is proportional to the number \(N(d,D)\) of
retained coefficients, while the bias is controlled by the
Fourier--Sobolev tail estimate~\eqref{eq:tail-bound}.  The following proposition records this
calculation.

\begin{proposition}[Projection risk bound]\label{prop:projection}
Let \(B>0\), let \(0\leqslant D\leqslant d\) be an integer, and let
\(V\) be an upper bound on \(\E Y^2\). Then the empirical Fourier
projection estimator satisfies
\begin{equation}\label{eq:projection-risk}
  \E\norm{\wt f_D-f}_2^2
  \leqslant
  \frac{B}{4(D+1)}
  +
  \frac{V N(d,D)}{n}
\end{equation}
for every \(f\) such that \(I^{(2)}(f)\leqslant B\). If, in addition,
\(f\) takes values in \([0,1]\), then the same bound holds for the
clipped estimator \(\wh f_D=\Pi_{[0,1]}\wt f_D\).

Moreover, one may take \(V=\sigma^2+B/4\) on the centered ellipsoid
\(\cE_B\), and \(V=\sigma^2+1\) for functions bounded by \([0,1]\).
\end{proposition}

The choice of the truncation degree is dictated by the two terms in
\eqref{eq:projection-risk}.  Increasing \(D\) decreases the bias, but
increases the variance through the number \(N(d,D)\) of retained Fourier
coefficients.  The resolvable Fourier degree \(D^*(d,n)\) is the largest
degree for which the full low-degree Fourier--Walsh space has dimension at
most \(n\).  To obtain a risk of order \(B/D^*(d,n)\), the variance term
must itself be no larger than this quantity.

Therefore, for the upper bound, it is convenient to use a cutoff adapted to
the target rate.  When \(D^*(d,n)\geqslant1\) and \(n\) is large enough for the set below
to be nonempty (see Lemma~\ref{lem:balanced-truncation}), define
\(D_\sharp\) as the largest degree not exceeding \(D^*(d,n)\) for which
the variance bound remains at the desired scale:
\begin{equation}\label{eq:Dsharp}
  D_\sharp=D_\sharp(d,n,B,V)
  =
  \max\left\{
    0\leqslant D\leqslant D^*(d,n):
    \frac{V N(d,D)}{n}
    \leqslant
    \frac{B}{4D^*(d,n)}
  \right\}.
\end{equation}
Thus \(D_\sharp\) is a slightly conservative version of the resolution
degree.  It keeps the variance at the same scale as the target rate while
remaining as close as possible to \(D^*(d,n)\).  In the regular
non-saturated regime considered in Theorem~\ref{thm:ellipsoid}, the
technical estimates in the proof show that this conservative cutoff is
still comparable to \(D^*(d,n)\).  It is therefore the cutoff used to
attain the upper bound in the main ellipsoid result.

\paragraph{Computation and tuning.}
The empirical Fourier projection estimator is explicit.  For a fixed
degree \(D\), it is obtained by computing the empirical coefficients
\[
  \wt f(S)=\frac1n\sum_{j=1}^n Y_j\chi_S(X_j),
  \qquad |S|\leqslant D.
\]
Thus the number of coefficients is
\[
  N(d,D)=\sum_{k=0}^D\binom dk .
\]
A direct implementation has computational cost of order
\(nN(d,D)D\), since each Walsh character of degree at most \(D\) can be
evaluated in at most \(D\) multiplications.  When the estimator is run
at the resolvable degree \(D=D^*(d,n)\), the defining inequality
\(N(d,D^*)\leqslant n\) gives a cost of order at most \(n^2 D^*\).
Moreover, whenever \(d\geqslant 2D^*\), one has
\(N(d,D^*)\geqslant\binom{d}{D^*}\geqslant 2^{D^*}\), so that
\(D^*\leqslant\log_2 n\). This condition holds under
Assumption~\ref{ass:regular}, where \(d\geqslant\kappa(D^*+1)\) with
\(\kappa>2\).  The cost is thus \(O(n^2\log n)\), with no explicit dependence on the ambient dimension
\(d\).  The degrees \(D^*\) and \(D_\sharp\) are found by increasing \(D\)
from \(0\) and updating \(N(d,D)\) incrementally, with \(D^*\) the
largest \(D\) such that \(N(d,D)\leqslant n\) and \(D_\sharp\) the
largest one satisfying \eqref{eq:Dsharp}.  This costs \(O(D^*)\)
arithmetic operations, negligible compared with forming the estimator.

The choice of \(D_\sharp\) depends on the radius \(B\) and on the moment
bound \(V\) through \eqref{eq:Dsharp}.  The upper bounds proved below
should therefore be read as oracle minimax bounds: they identify the
resolution scale and the corresponding risk when \(B\) and \(V\) are
treated as known.  Data-driven choices of the truncation degree, for
instance by model selection or Lepski-type methods, are not pursued here.

\paragraph{The saturated regime.}
When \(n\geqslant 2^d\), the sample size is at least the number of
Fourier--Walsh coefficients: here \(N(d,d)=2^d\leqslant n\), so that
\(D^*(d,n)=d\) and the degree-\(d\) projection retains the full
Fourier--Walsh expansion.  Equivalently, the Fourier tail is empty:
\[
  \sum_{|S|>d}\wh f(S)^2=0.
\]
The projection at \(D=d\) is therefore unbiased, and the risk reduces to
its variance term.  For any class on which \(\E Y^2\leqslant V\),
\[
  \sup_f R(\wt f_d,f)
  \leqslant
  \frac{V2^d}{n},
\]
where the supremum is over that class.  In this regime the problem reduces
to estimating a vector of length \(2^d\), and the natural scale is the
finite-dimensional parametric rate \(2^d/n\).

We now turn to the non-saturated high-dimensional regime.  The main
ellipsoid theorem is stated in the regular resolution window of
Assumption~\ref{ass:regular}, where \(D^*(d,n)\to\infty\) while remaining
separated from the ambient dimension \(d\).  In that regime the relevant
quantity is not the total number of vertices of the cube, but the largest
complete Fourier degree that can be resolved from the sample.

\section{Main results}\label{sec:main-results}

\subsection{The Fourier--Sobolev ellipsoid}

We now state the central result.  The theorem is non-asymptotic, but its
conditions are most naturally interpreted in an asymptotic window where
the resolvable Fourier degree tends to infinity while remaining well
below the ambient dimension.

\begin{assumption}[Regular resolution window]\label{ass:regular}
Along a sequence of pairs \((d,n)\), with \(d\) allowed to depend on
\(n\), let \(D^*=D^*(d,n)\).  We assume that
\[
  D^*\to\infty
\]
and that, for some fixed constant \(\kappa>2\),
\begin{equation}\label{eq:regular-window}
  d\geqslant \kappa(D^*+1)
\end{equation}
for all sufficiently large \(n\).  The signal budget \(B>0\) and the
sub-Gaussian noise parameter \(\sigma^2>0\) are fixed.
\end{assumption}

The condition \eqref{eq:regular-window} has a simple purpose.  It keeps the
resolved degrees on the increasing side of the binomial profile, well
before the middle layer of the cube.  In this range,
Lemma~\ref{lem:growth} shows that consecutive Fourier layers grow at a
controlled geometric rate.  This prevents the first unresolved layer from
being negligible relative to the cumulative number of resolved
coefficients, which is the comparison needed in the lower bound.

The regular-resolution condition holds in the
logarithmic growth regime
\[
  \log d=(\log n)^b,
  \qquad 0<b<1.
\]
Here \(d\) grows faster than any power of \(\log n\) but slower than any
power of \(n\), so that \(2^d\) exceeds every power of \(n\) while
remaining below \(2^n\); in particular the hypercube is far from
saturated.  In this regime the resolvable degree can be computed
explicitly up to constants, as stated after the theorem.

\begin{theorem}[Minimax rate over the Fourier--Sobolev ellipsoid]\label{thm:ellipsoid}
Let \(\kappa>2\) be fixed, and let \(\cE_B\) be the centered
Fourier--Sobolev ellipsoid defined in \eqref{eq:ellipsoid}.  Then there
exist universal constants \(0<c<C<\infty\) and an integer \(D_0\geqslant1\) depending only on \(B\), \(\sigma^2\), and \(\kappa\),
such that, for every pair \((d,n)\) satisfying
\(d\geqslant\kappa(D^*(d,n)+1)\) and
\(D^*(d,n)\geqslant D_0\),
\[
  c\,\frac{B}{D^*(d,n)}
  \leqslant
  R_n^*(\cE_B)
  \leqslant
  C\,\frac{B}{D^*(d,n)}.
\]
In particular, under Assumption~\ref{ass:regular}, both conditions are
satisfied for all sufficiently large \(n\), and the two-sided bound is
therefore valid.

The upper bound is attained by the empirical Fourier projection estimator
\eqref{eq:estimator-intro}, with truncation degree \(D_\sharp\) defined in
\eqref{eq:Dsharp} using \(V=\sigma^2+B/4\).
\end{theorem}

The upper bound follows from Proposition~\ref{prop:projection}.  Indeed,
whenever \(d\geqslant\kappa(D^*(d,n)+1)\) and \(D^*(d,n)\)
exceeds the threshold of Lemma~\ref{lem:balanced-truncation}, the cutoff
\(D_\sharp\) in \eqref{eq:Dsharp} is well defined and satisfies
\[
  D_\sharp+1
  \geqslant
  \frac12 D^*(d,n).
\]
Applying Proposition~\ref{prop:projection} with \(D=D_\sharp\) and
\(V=\sigma^2+B/4\) gives
\[
  \sup_{f\in\cE_B} R(\wt f_{D_\sharp},f)
  \leqslant
  \frac{B}{4(D_\sharp+1)}
  +
  \frac{B}{4D^*(d,n)}
  \leqslant
  \frac{3B}{4D^*(d,n)}.
\]

Theorem~\ref{thm:ellipsoid} expresses the minimax rate through the
resolvable degree \(D^*(d,n)\).  The following corollary spells out this
quantity, and hence the resulting rate, in the logarithmic growth regime
mentioned above.

\begin{corollary}[Logarithmic growth regime]\label{cor:regimes}
Assume that \(\log d=(\log n)^b\) for some \(0<b<1\).  Then Assumption~\ref{ass:regular} holds for all
sufficiently large \(n\), and
\[
  D^*(d,n)
  \asymp
  \frac{\log n}{\log d}
  \asymp
  (\log n)^{1-b}.
\]
Consequently,
\[
  R_n^*(\cE_B)
  \asymp
  B\frac{\log d}{\log n}
  \asymp
  B(\log n)^{-(1-b)}.
\]
The implicit constants in these equivalences, and the threshold in
\(n\) beyond which they hold, may depend on \(b\), \(B\), and
\(\sigma^2\).
\end{corollary}

This regime interpolates between two extremes.  If \(b\) is close to zero,
then \(\log d\) grows slowly compared with \(\log n\), the resolvable degree
is close to logarithmic in \(n\), and the rate is close to \(B/\log n\).
As \(b\) increases toward one, the ambient dimension grows faster, the
resolvable degree grows more slowly, and the rate deteriorates.  The endpoint \(b=1\) is outside 
the regular resolution window considered in
Theorem~\ref{thm:ellipsoid}.  In that case the heuristic value
\(\log n/\log d\) no longer diverges, and the mechanism of the theorem does
not yield a vanishing risk over the full ellipsoid.

\subsection{Monotone models with Fourier--Sobolev budgets}

We now turn to monotone regression functions.  In this
section, monotone functions are assumed to take values in \([0,1]\).  This
boundedness condition fixes the scale of the monotone model and also allows
us to use the clipped projection estimator. For \(B>0\), define
\begin{equation}
\label{eq:Mdb}
  \cM_{d,B}
  =
  \left\{
    f:\Hd\to[0,1]:
    f \text{ is coordinatewise nondecreasing and }
    I^{(2)}(f)\leqslant B
  \right\}.
\end{equation}
The projection bound applies immediately to this class.  Indeed, the range
condition gives the moment bound \(V=\sigma^2+1\), while the quadratic
Fourier--Sobolev budget controls the projection bias.  Thus, under Assumption~\ref{ass:regular}, for all sufficiently large \(n\),
the clipped empirical Fourier projection estimator satisfies
\[
  \sup_{f\in\cM_{d,B}}R(\wh f_{D_\sharp},f)
  \leqslant
  C\,\frac{B}{D^*(d,n)},
\]
where \(D_\sharp\) is defined in \eqref{eq:Dsharp} with
\(V=\sigma^2+1\). This upper bound is a direct consequence of
Proposition~\ref{prop:projection}; it does not use monotonicity, except
through the range constraint.

\subsection*{Known-support monotone model}

We first consider the oracle case in which the relevant coordinates are
known.  Without loss of generality, we identify them with the first \(s\)
coordinates, so that the regression function is defined on the
\(s\)-dimensional cube \(\Hs\). Accordingly, throughout this subsection, risks are
computed with respect to the uniform design on \(\Hs\).

Let \(s=s_n\) be an integer sequence satisfying, for some fixed constants
\(\varepsilon>0\) and \(C_0<\infty\),
\begin{equation}\label{eq:known-support-window}
  (1+\varepsilon)\log_2 n
  \leqslant
  s
  \leqslant
  C_0\log_2 n.
\end{equation}
The lower bound is the essential one: it ensures \(2^s\geqslant
n^{1+\varepsilon}\), so that the active subcube is far from saturated,
while the upper bound only keeps \(s\) of logarithmic order. Under the growth condition
\eqref{eq:known-support-window}, Assumption~\ref{ass:regular}
holds on the \(s\)-dimensional cube (as shown in the Supplementary Material)  and
\[
  D^*(s,n)\asymp s\asymp \log n .
\]
Thus, on the known \(s\)-dimensional cube, the Fourier--Sobolev projection
rate \(B/D^*(s,n)\) is of order \(B/s\).

On \(\Hs\), define the monotone Fourier--Sobolev class
\begin{equation}\label{eq:known-support-class}
  \cG_{s,B}
  =
  \left\{
    g:\Hs\to[0,1]:
    g\text{ is coordinatewise nondecreasing and }
    I^{(2)}(g)\leqslant B
  \right\}.
\end{equation}
This is exactly the ambient class \(\cM_{s,B}\) of \eqref{eq:Mdb} with
\(d=s\); we use the notation \(\cG_{s,B}\) to emphasize that it lives on
the active subcube.

\begin{theorem}[Known-support monotone model]\label{thm:known-support}
Assume that the growth condition \eqref{eq:known-support-window} holds,
and let \(\cG_{s,B}\) be the monotone Fourier--Sobolev class defined in
\eqref{eq:known-support-class}.  Then there exist constants \(0<c<C<\infty\), depending only on
\(\varepsilon\) and \(C_0\), and an integer \(n_0\), depending only on
\(B\), \(\sigma^2\), \(\varepsilon\), and \(C_0\), such that, for all
\(n\geqslant n_0\),
\[
  c\,\frac{B}{s}
  \leqslant
  R_n^*(\cG_{s,B})
  \leqslant
  C\,\frac{B}{s}.
\]
The upper bound is attained by the clipped empirical Fourier projection
estimator on \(\Hs\), with truncation degree \(D_\sharp\) defined in
\eqref{eq:Dsharp} using \(d=s\) and \(V=\sigma^2+1\).
\end{theorem}

The lower bound uses the monotone structure in an essential way.  The
packing is placed on the middle layer of \(\Hs\), where changing the
value at one point imposes no monotonicity constraint within the layer
itself.  Each pattern on that layer is then extended to a monotone function on
all of \(\Hs\), constant below the layer and constant above it. The Fourier--Sobolev 
budget fixes the admissible amplitude of the
construction.

\subsection*{Intrinsic unknown-support monotone model}

We now separate the ambient dimension from the intrinsic monotone dimension.
For \(f:\Hd\to\mathbb R\), define the active support by
\[
  \operatorname{supp}(f)
  =
  \{i\in[d]:\Delta_i f\not\equiv0\}.
\]
For \(1\leqslant s\leqslant d\), set
\[
  \cM_{d,B,s}
  =
  \left\{
    f\in\cM_{d,B}:
    |\operatorname{supp}(f)|\leqslant s
  \right\}.
\]
The support is not assumed to be known.  The following result shows that
adaptation to the unknown support costs only the logarithmic price of
selecting the active coordinates.

\begin{theorem}[Intrinsic monotone model with unknown support]\label{thm:intrinsic-monotone}
Assume that the growth condition \eqref{eq:known-support-window} holds,
and let \(\cM_{d,B,s}\) be the intrinsic monotone class with unknown
support defined above.  Then there exist constants \(0<c<C<\infty\), depending only on
\(\sigma^2\), \(\varepsilon\), and \(C_0\), and an integer \(n_0\),
depending only on \(B\), \(\sigma^2\), \(\varepsilon\), and \(C_0\), such
that, for all \(n\geqslant n_0\),
\[
  c\,\frac{B}{s}
  \leqslant
  R_n^*(\cM_{d,B,s})
  \leqslant
  C\left(
  \frac{B}{s}
  +
  \frac{\log\binom ds}{n}
  \right).
\]
The upper bound is attained by a model-selection projection estimator
constructed in the Supplementary Material.

Consequently, if
\[
  \log\binom ds
  \leqslant
  A_1\frac{nB}{s}
\]
for some fixed constant \(A_1<\infty\), then
\[
  R_n^*(\cM_{d,B,s})
  \asymp
  \frac{B}{s}.
\]
\end{theorem}

Theorem~\ref{thm:intrinsic-monotone} has a direct interpretation.  When
the active dimension satisfies \(s\asymp\log n\), the minimax risk over the
known-support monotone Fourier--Sobolev class is of order \(B/s\), and
hence of order \(1/\log n\) when \(B\) is fixed.  When the support is
unknown, the same intrinsic term remains present, while an additional
complexity term of order \(\log\binom ds/n\) appears.  This term is the model-selection price for choosing the active
coordinates among the \(\binom ds\) possible supports of size \(s\).

The condition
\[
  \log\binom ds \leqslant A_1\frac{nB}{s}
\]
has a simple interpretation.  Using the standard bound
\[
  \log\binom ds
  \leqslant
  s\log\left(\frac{ed}{s}\right),
\]
it is enough that
\[
  \log\left(\frac{ed}{s}\right)
  \lesssim
  \frac{nB}{s^2}.
\]
Since
\[
  \log\left(\frac{ed}{s}\right)
  =
  1+\log d-\log s
  \leqslant
  1+\log d,
\]
a simple sufficient condition when \(s\asymp\log n\) is
\[
  1+\log d
  \lesssim
  \frac{nB}{(\log n)^2}.
\]
In particular, polynomial ambient dimensions \(d=n^a\), with fixed
\(a>0\), are allowed here.  Unlike the ellipsoid regime of
Corollary~\ref{cor:regimes}, the intrinsic monotone model accommodates
polynomially large ambient dimension, since \(d\) enters only through
the support-selection term \(\log\binom ds\) and not through the Fourier
resolution, which is governed by \(s\).  Under such growth conditions,
the support-selection cost is no larger than the intrinsic rate
\(B/s\), and the unknown-support model has the same minimax rate as the
known-support model.

The estimator attaining the upper bound is a model-selection projection
estimator, described in the Supplementary Material. For each
candidate support of size \(s\), it computes a clipped empirical Fourier
projection on the corresponding \(s\)-dimensional subcube.  The truncation
degree is chosen at the intrinsic resolution scale, governed by \(s\)
rather than by the ambient dimension \(d\).  Since \(D^*(s,n)\asymp s\) in
the regime considered here, the projection part of the risk is of order
\(B/s\).  The final selection over supports contributes the additional term
\(\log\binom ds/n\).  Thus the intrinsic model separates two effects that
are confounded in the ambient problem: the statistical resolution of
Fourier--Walsh interactions on the active subcube and the combinatorial  
cost of finding that subcube.  The same combinatorial effect appears on
the computational side, where the estimator searches over all
\(\binom ds\) candidate supports and its running time grows
exponentially in the intrinsic dimension \(s\).

\section{Discussion and connections}\label{sec:discussion-related}

\paragraph{Intrinsic dimension.}
Throughout, the monotone results are formulated through the
Fourier--Sobolev energy \(I^{(2)}\), which controls the tail of the
Fourier--Walsh expansion, while monotonicity plays a complementary
geometric role by enabling the middle-layer packing in the lower
bounds.  The intrinsic monotone model separates the ambient dimension
\(d\) from the number \(s\) of active coordinates.  The full ellipsoid is governed by the
ambient resolution \(D^*(d,n)\).  By contrast, a monotone function depending
on at most \(s\) coordinates is governed by the resolution of an
\(s\)-dimensional subcube.  When the support is unknown, this intrinsic
rate is supplemented by the model-selection cost \(\log\binom ds/n\). Thus, 
in the logarithmic regime \(s\asymp\log n\), the leading term is
\(B/s\), while the additional term reflects the cost of identifying the
active coordinates.

\paragraph{Local entropy and convex constraints.}
General minimax theory based on local metric entropy gives powerful
characterizations for convex and star-shaped classes
\citep{yang1999,shrotriya2022lecam,neykov2023,prasadan2025reg,prasadan2025}.
Convex-constrained least squares has also been analyzed through Gaussian
processes and related local geometric quantities; see, for example,
\citet{chatterjee2014}.  The present setting is consistent with this
viewpoint, but the Fourier--Walsh basis makes the relevant geometry
explicit.  The class decomposes into Fourier layers of known cardinality,
and the effective dimension is determined by the number of coefficients
below the resolution boundary.  This leads to rates summarized by the
resolvable degree \(D^*(d,n)\).

\paragraph{Low-degree Boolean polynomial learning.}
There is also a related literature on learning Boolean functions and
polynomial surrogates from low-degree Fourier information.
\citet{eskenazis2022} show that bounded functions of bounded degree are
learnable in \(L_2\) from a number of random queries only logarithmic in
the dimension, and \citet{vandoornmalen2026} determine the sample
complexity of uniform \(L_\infty\) estimation for bounded low-degree and
sparse Fourier--Walsh polynomials.  Both fix the degree or sparsity in
advance.  Here, by contrast, the degree is selected by the sample size
through the resolution boundary, and the target class is a
Fourier--Sobolev ellipsoid whose tail is controlled by the soft budget
\(I^{(2)}\) rather than a hard degree cut-off.  A separate line of work
studies the recovery of Fourier-sparse set functions from random
evaluations \citep{stobbe2012}. There the structural assumption is exact
sparsity of the Fourier support and the goal is exact recovery, whereas
we impose a soft budget on the Fourier--Sobolev energy and characterize
the minimax estimation rate under noise.

\paragraph{Boolean Fourier analysis and high-order interactions.}
Our notation and basic facts follow the standard Fourier analysis of
Boolean functions \citep{odonnell2014}.  Classical results such as the KKL
theorem \citep{kahn1988} and the Benjamini--Kalai--Schramm
noise-sensitivity theorem \citep{bks1999} illustrate the central role of
Fourier mass and discrete derivatives in this theory.  The Fourier
coefficients on the Boolean cube may also be interpreted as interaction
coefficients.  This viewpoint is common in applications where
\(x\in\{0,1\}^d\) indexes combinations of binary factors and high-degree
Fourier terms represent high-order epistasis; see \citet{weinreich2013}.
The Walsh--Hadamard basis has also been generalized to non-uniform
probability measures on the hypercube, with connections to explainability
methods \citep{ferrere2025}.
From this perspective, \(D^*(d,n)\) is the highest interaction order whose
complete coefficient set can be estimated from \(n\) random samples.

\appendix
\section{Appendix: Main proofs}\label{sec:proofs}

\subsection{Proof of Proposition~\ref{prop:projection}}

Let \((X,Y)\) denote a generic observation distributed as each
\((X_j,Y_j)\).  Thus \(X\sim\mu\), \(Y=f(X)+\varepsilon\), and all
expectations involving \((X,Y)\) are taken with respect to this joint law.
Expectations involving \(\wt f_D\) are taken over the full sample.

Under the model \eqref{eq:model}, for every \(S\subseteq[d]\), the
empirical coefficient is unbiased:
\[
  \E\wt f(S)
  =
  \E\big[Y\chi_S(X)\big]
  =
  \E\big[f(X)\chi_S(X)\big]
  =
  \wh f(S).
\]
Moreover,
\[
  \Var(\wt f(S))
  =
  \frac{1}{n}\Var\big(Y\chi_S(X)\big)
  \leqslant
  \frac{1}{n}\E\big[Y^2\chi_S(X)^2\big]
  \leqslant
  \frac{V}{n}.
\]
Since the Walsh functions form an orthonormal basis of \(L^2(\Hd,\mu)\),
for each realization of the sample,
\[
  \norm{\wt f_D-f}_2^2
  =
  \sum_{|S|\leqslant D}\big(\wt f(S)-\wh f(S)\big)^2
  +
  \sum_{|S|>D}\wh f(S)^2.
\]
Taking expectation over the sample gives
\[
  \E\norm{\wt f_D-f}_2^2
  =
  \sum_{|S|\leqslant D}\E\big(\wt f(S)-\wh f(S)\big)^2
  +
  \sum_{|S|>D}\wh f(S)^2.
\]
The first term is bounded by \(VN(d,D)/n\), and the second by the
Fourier--Sobolev tail estimate \eqref{eq:tail-bound}.  This proves
\eqref{eq:projection-risk}.

If \(f\) takes values in \([0,1]\), then pointwise projection onto
\([0,1]\) is a contraction:
\[
  \big|\Pi_{[0,1]}(u)-f(x)\big|
  \leqslant
  |u-f(x)|
  \qquad
  \text{for all }u\in\mathbb R,\ x\in\Hd.
\]
Therefore clipping cannot increase the \(L^2(\mu)\) error, and the same
bound holds for \(\wh f_D=\Pi_{[0,1]}\wt f_D\).

It remains to justify the stated choices of \(V\).  On \(\cE_B\), one has
\(\wh f(\emptyset)=0\).  Hence, by orthonormality of the Walsh basis,
\[
  \E f(X)^2
  =
  \norm{f}_2^2
  =
  \sum_{S\ne\emptyset}\wh f(S)^2
  \leqslant
  \sum_{S\ne\emptyset}|S|\wh f(S)^2.
\]
By the definition of the Fourier--Sobolev energy,
\[
  \sum_{S\ne\emptyset}|S|\wh f(S)^2
  =
  \frac{I^{(2)}(f)}{4}
  \leqslant
  \frac{B}{4}.
\]
Thus \(\E f(X)^2\leqslant B/4\).  We also note that the sub-Gaussian
condition implies \(\E\varepsilon^2\leqslant\sigma^2\): since
\(\cosh x\geqslant1+x^2/2\) for all real \(x\), one has, for every
\(t\neq0\),
\[
  1+\frac{t^2\,\E\varepsilon^2}{2}
  \leqslant
  \frac{\E e^{t\varepsilon}+\E e^{-t\varepsilon}}{2}
  \leqslant
  e^{\sigma^2t^2/2}.
\]
Subtracting \(1\) and dividing by \(t^2/2\) on both sides gives
\(\E\varepsilon^2\leqslant\sigma^2(e^{u}-1)/u\) with
\(u=\sigma^2t^2/2\), and letting \(t\to0\) yields
\(\E\varepsilon^2\leqslant\sigma^2\).  Since \(Y=f(X)+\varepsilon\), with \(\varepsilon\)
independent of \(X\) and \(\E\varepsilon=0\), we have
\[
  \E Y^2
  =
  \E f(X)^2+\E\varepsilon^2
  \leqslant
  \frac{B}{4}+\sigma^2.
\]
Thus \(V=\sigma^2+B/4\) is valid on \(\cE_B\).  If
\(0\leqslant f\leqslant1\), then \(\E f(X)^2\leqslant1\), and therefore
\(V=\sigma^2+1\) is valid.

\subsection{Elementary tools}

Under Assumption~\ref{ass:regular}, the relevant degrees lie before the
middle of the binomial profile.  We shall use the following elementary
consequence of this fact.  The point is that, in this range, consecutive
Fourier layers grow at a controlled geometric rate.  As a result, the first
unresolved layer is comparable to the cumulative number of resolved
coefficients.

\begin{lemma}[Growth of Fourier layers]\label{lem:growth}
Let \(D\geqslant1\) be an integer, and suppose that
\[
  d\geqslant \kappa(D+1)
\]
for some fixed \(\kappa>2\).  Then \(D+1\leqslant d\), and there exist
constants \(a_\kappa>1\), \(c_\kappa>0\), and \(C_\kappa<\infty\),
depending only on \(\kappa\), such that the following hold.
\begin{enumerate}[label=(\roman*),leftmargin=2.2em]
\item For all \(1\leqslant k\leqslant D\),
\[
  N(d,k)\geqslant a_\kappa N(d,k-1).
\]

\item If \(m=D+1\), then
\[
  \sum_{j=0}^{m-1}\binom{d}{j}
  \leqslant
  C_\kappa \binom{d}{m}.
\]

\item Consequently, if \(D=D^*(d,n)\) and \(m=D+1\), then
\[
  \binom{d}{m}\geqslant c_\kappa n.
\]
\end{enumerate}
\end{lemma}

\begin{proof}
Since \(d\geqslant \kappa(D+1)\) and \(\kappa>2\), we have
\(D+1\leqslant d\).  Choose a constant \(q_\kappa\) such that
\[
  \frac{1}{\kappa-1}<q_\kappa<1.
\]
For every \(1\leqslant j\leqslant D+1\),
\[
  \frac{\binom d{j-1}}{\binom dj}
  =
  \frac{j}{d-j+1}
  \leqslant
  \frac{D+1}{d-D}
  \leqslant
  \frac{D+1}{\kappa(D+1)-D}
  \leqslant
  \frac{1}{\kappa-1}
  <q_\kappa .
\]
Thus the binomial layers are geometrically increasing up to level
\(D+1\).  In particular, for \(0\leqslant j<k\leqslant D+1\),
\[
  \binom dj
  \leqslant
  q_\kappa^{\,k-j}\binom dk .
\]

Taking \(1\leqslant k\leqslant D\) and summing over
\(0\leqslant j\leqslant k-1\), we obtain
\[
  N(d,k-1)
  =
  \sum_{j=0}^{k-1}\binom dj
  \leqslant
  \binom dk
  \sum_{r=1}^{k}q_\kappa^r
  \leqslant
  \frac{q_\kappa}{1-q_\kappa}\binom dk .
\]
Therefore
\[
  \binom dk
  \geqslant
  \frac{1-q_\kappa}{q_\kappa}N(d,k-1),
\]
and hence
\[
  N(d,k)
  =
  N(d,k-1)+\binom dk
  \geqslant
  \frac{1}{q_\kappa}N(d,k-1).
\]
This proves (i), with \(a_\kappa=q_\kappa^{-1}>1\).

The same geometric summation with \(k=m=D+1\) gives
\[
  \sum_{j=0}^{m-1}\binom dj
  \leqslant
  \frac{q_\kappa}{1-q_\kappa}\binom dm ,
\]
so (ii) holds with \(C_\kappa=q_\kappa/(1-q_\kappa)\).

Finally, suppose that \(D=D^*(d,n)\) and set \(m=D+1\).  Since
\(m\leqslant d\), maximality of \(D^*(d,n)\) gives
\[
  N(d,D)\leqslant n<N(d,m).
\]
On the other hand, (ii) implies
\[
  N(d,m)
  =
  \sum_{j=0}^{m}\binom dj
  \leqslant
  (1+C_\kappa)\binom dm .
\]
Therefore
\[
  \binom dm
  \geqslant
  \frac{1}{1+C_\kappa}\,n.
\]
This proves (iii), with \(c_\kappa=(1+C_\kappa)^{-1}\).
\end{proof}

The next lemma shows that the balanced truncation degree remains comparable
to the resolvable degree.

\begin{lemma}[Balanced truncation]\label{lem:balanced-truncation}
Let \(B>0\) and \(V>0\) be fixed, and let \(\kappa>2\).  There exists an
integer \(D_1\geqslant1\), depending only on \(B\), \(V\), and \(\kappa\), such that, 
for every pair \((d,n)\) satisfying
\(d\geqslant\kappa(D^*(d,n)+1)\) and
\(D^*(d,n)\geqslant D_1\), the degree
\(D_\sharp\) defined in \eqref{eq:Dsharp} is well defined and satisfies
\begin{equation}\label{eq:Dsharp-close}
  D_\sharp+1
  \geqslant
  \frac12 D^*(d,n).
\end{equation}
\end{lemma}

\begin{proof}
Write \(D^*=D^*(d,n)\geqslant1\) and
\[
  \tau=\frac{B}{4VD^*}.
\]

If
\[
  N(d,D^*)\leqslant \tau n,
\]
then \(D^*\) belongs to the admissible set in the definition of
\(D_\sharp\).  Hence \(D_\sharp=D^*\), and there is nothing to prove.

We may therefore assume that
\[
  N(d,D^*)>\tau n.
\]
In particular, \(\tau<1\), since \(N(d,D^*)\leqslant n\) by definition of
\(D^*\).  Choose
\[
  r=
  \left\lceil
  \frac{\log(1/\tau)}{\log a_\kappa}
  \right\rceil,
\]
where \(a_\kappa>1\) is the constant in Lemma~\ref{lem:growth}.  Since
\[
  \log(1/\tau)
  =
  \log D^*+\log(4V/B),
\]
we have
\[
  r\leqslant 1+\frac{\log D^*+\log(4V/B)}{\log a_\kappa}.
\]
Since the right-hand side grows only logarithmically in \(D^*\), there
exists an integer \(D_1\), depending only on \(B\), \(V\), and
\(\kappa\), such that \(r\leqslant D^*/2\) whenever
\(D^*\geqslant D_1\).  In particular,
\[
  1\leqslant D^*-r+1\leqslant D^*.
\]

We then apply Lemma~\ref{lem:growth}(i) successively with
\[
  k=D^*-r+1,\ldots,D^*.
\]
This gives
\[
  N(d,D^*-r)
  \leqslant
  a_\kappa^{-r}N(d,D^*).
\]
By the definition of \(r\), \(a_\kappa^{-r}\leqslant \tau\), and by the
definition of \(D^*\), \(N(d,D^*)\leqslant n\).  Hence
\[
  N(d,D^*-r)
  \leqslant
  \tau n.
\]
Thus \(D^*-r\) belongs to the admissible set in the definition of
\(D_\sharp\), and therefore
\[
  D_\sharp\geqslant D^*-r.
\]
Consequently,
\[
  D_\sharp+1
  \geqslant
  D^*-r+1
  \geqslant
  \frac{D^*}{2}.
\]
This proves \eqref{eq:Dsharp-close}.
\end{proof}

We also record the form of Fano's inequality used in the lower bounds.

\begin{lemma}[Fano reduction]\label{lem:fano}
Let \(\rho\) be a metric and let \(\{f_\theta:\theta\in\Theta\}\) be a
finite family.  Suppose that there exists \(\theta_0\in\Theta\) such that,
with
\[
  \Theta_1=\Theta\setminus\{\theta_0\},
  \qquad
  M=|\Theta_1|,
\]
one has \(M\geqslant2\) and
\[
  \rho(f_\theta,f_{\theta'})
  \geqslant
  2\delta
  \qquad
  \text{for all distinct }\theta,\theta'\in\Theta_1.
\]
Assume that
\begin{equation}\label{eq:fano-base}
  \frac{1}{M}\sum_{\theta\in\Theta_1}
  \KL(P_\theta,P_{\theta_0})
  \leqslant
  \alpha \log M
\end{equation}
for some \(0<\alpha<1\).  Then
\begin{equation}\label{eq:fano-conclusion}
  \inf_{\tilde f}\sup_{\theta\in\Theta}
  \E_\theta\rho^2(\tilde f,f_\theta)
  \geqslant
  \delta^2
  \left(
    1-\alpha-\frac{\log 2}{\log M}
  \right).
\end{equation}
In particular, if
\[
  1-\alpha-\frac{\log 2}{\log M}>0,
\]
then the minimax risk over the family is bounded below by a positive
constant multiple of \(\delta^2\).

A sufficient condition for \eqref{eq:fano-base} is the pairwise bound
\begin{equation}\label{eq:fano-pair}
  \sup_{\substack{\theta,\theta'\in\Theta\\ \theta\ne\theta'}}
  \KL(P_\theta,P_{\theta'})
  \leqslant
  \alpha \log M .
\end{equation}
\end{lemma}

\begin{proof}
Since the supremum over \(\Theta\) is at least the supremum over
\(\Theta_1\),
\[
  \inf_{\tilde f}\sup_{\theta\in\Theta}
  \E_\theta\rho^2(\tilde f,f_\theta)
  \geqslant
  \inf_{\tilde f}\sup_{\theta\in\Theta_1}
  \E_\theta\rho^2(\tilde f,f_\theta).
\]
The family indexed by \(\Theta_1\) is \(2\delta\)-separated in the metric
\(\rho\).  We first reduce estimation to testing.  Fix an estimator
\(\tilde f\) and let
\(\psi\in\arg\min_{\theta\in\Theta_1}\rho(\tilde f,f_\theta)\)
be a minimum-distance test, with ties broken arbitrarily.  If
\(\rho(\tilde f,f_\theta)<\delta\) for the true
\(\theta\in\Theta_1\), then, for every
\(\theta'\in\Theta_1\setminus\{\theta\}\), the triangle inequality
gives
\[
  \rho(\tilde f,f_{\theta'})
  \geqslant
  \rho(f_\theta,f_{\theta'})-\rho(\tilde f,f_\theta)
  >
  2\delta-\delta
  =
  \delta
  >
  \rho(\tilde f,f_\theta),
\]
so that \(\psi=\theta\).  Hence, for every \(\theta\in\Theta_1\),
\(P_\theta(\psi\neq\theta)\leqslant
P_\theta(\rho(\tilde f,f_\theta)\geqslant\delta)\), and, by Markov's
inequality,
\[
  \sup_{\theta\in\Theta_1}
  \E_\theta\rho^2(\tilde f,f_\theta)
  \geqslant
  \delta^2
  \max_{\theta\in\Theta_1}
  P_\theta(\psi\neq\theta).
\]
Next, we lower bound the testing error. Let
\(\bar P=M^{-1}\sum_{\theta\in\Theta_1}P_\theta\) and
\(\mathcal I=M^{-1}\sum_{\theta\in\Theta_1}\KL(P_\theta,\bar P)\).
Fano's inequality \citep[Lemma~2.10]{tsybakov2009}, applied to a uniform
prior on \(\Theta_1\), gives, for every test \(\psi\) with values in
\(\Theta_1\),
\[
  \max_{\theta\in\Theta_1}P_\theta(\psi\neq\theta)
  \geqslant
  \frac1M\sum_{\theta\in\Theta_1}P_\theta(\psi\neq\theta)
  \geqslant
  1-\frac{\mathcal I+\log2}{\log M}.
\]
Moreover, for every probability measure \(Q\),
\[
  \frac1M\sum_{\theta\in\Theta_1}\KL(P_\theta,Q)
  =
  \mathcal I+\KL(\bar P,Q)
  \geqslant
  \mathcal I,
\]
an identity that follows directly from the chain rule for
relative entropy. Choosing
\(Q=P_{\theta_0}\) and invoking \eqref{eq:fano-base} yields
\(\mathcal I\leqslant\alpha\log M\).  Combining the last three
displays, which hold for every estimator \(\tilde f\), and taking the
infimum over \(\tilde f\), proves \eqref{eq:fano-conclusion}.

Finally, if \eqref{eq:fano-pair} holds, then for any fixed
\(\theta_0\in\Theta\),
\[
  \frac{1}{M}\sum_{\theta\in\Theta_1}
  \KL(P_\theta,P_{\theta_0})
  \leqslant
  \alpha\log M,
\]
which is precisely \eqref{eq:fano-base}.
\end{proof}

\subsection{Proof of Theorem~\ref{thm:ellipsoid}}

We prove the lower bound by constructing a packing supported on a single
Fourier layer.  Put
\[
  D^*=D^*(d,n),
  \qquad
  m=D^*+1,
  \qquad
  \mathcal L_m=\{S\subseteq[d]: |S|=m\},
  \qquad
  L_m=|\mathcal L_m|=\binom{d}{m}.
\]
Since \(D^*(d,n)\geqslant D_0\geqslant1\) and
\(d\geqslant\kappa(D^*(d,n)+1)\), Lemma~\ref{lem:growth}(iii)
applies with \(D=D^*(d,n)\), and gives \(L_m\geqslant c_\kappa n\).

We now construct a finite family of functions supported on the single
Fourier layer \(\mathcal L_m\).  An element of
\(\{-1,1\}^{\mathcal L_m}\) is a choice of one sign for each set
\(S\in\mathcal L_m\).  Thus, if
\(\alpha\in\{-1,1\}^{\mathcal L_m}\), we write
\[
  \alpha=(\alpha_S)_{S\in\mathcal L_m},
  \qquad
  \alpha_S\in\{-1,1\}.
\]
We equip this hypercube of sign vectors with the Hamming distance
\[
  d_H(\alpha,\beta)
  =
  \big|\{S\in\mathcal L_m:\alpha_S\ne\beta_S\}\big|.
\]
By the Varshamov--Gilbert bound \citep[Lemma~4.7]{massart2007}, there
exist universal constants \(c_1>0\) and \(c_2>0\), and a subset
\(\mathcal A\subset\{-1,1\}^{\mathcal L_m}\), such that
\[
  \log |\mathcal A|\geqslant c_1 L_m
\]
and
\[
  d_H(\alpha,\beta)\geqslant c_2L_m
  \qquad
  \text{for all distinct }\alpha,\beta\in\mathcal A.
\]
Thus \(\mathcal A\) contains exponentially many sign patterns on
\(\mathcal L_m\), any two of which differ on a fixed positive fraction of
the coordinates.

For \(a>0\) and \(\alpha\in\mathcal A\), define
\[
  f_\alpha(x)
  =
  a\sum_{S\in\mathcal L_m}\alpha_S\chi_S(x).
\]
Then the only nonzero Fourier coefficients of \(f_\alpha\) are those on
the layer \(\mathcal L_m\), namely
\[
  \wh f_\alpha(S)
  =
  \begin{cases}
    a\alpha_S, & S\in\mathcal L_m,\\
    0, & S\notin\mathcal L_m.
  \end{cases}
\]
In particular, since \(m\geqslant1\), its constant coefficient is zero:
\[
  \wh f_\alpha(\emptyset)=0.
\]
Choose
\[
  a^2=\frac{B}{4mL_m}.
\]
Using the Fourier representation of the Fourier--Sobolev energy,
\[
  I^{(2)}(f)
  =
  4\sum_{S\ne\emptyset}|S|\wh f(S)^2,
\]
we obtain
\[
  I^{(2)}(f_\alpha)
  =
  4\sum_{S\in\mathcal L_m}m\,a^2
  =
  4m a^2 L_m
  =
  B.
\]
Therefore \(f_\alpha\in\cE_B\).  The zero function also belongs to
\(\cE_B\).

The set \(\mathcal A\) is used in two ways.  Its separation in Hamming
distance gives separation in \(L^2(\mu)\), while its cardinality will make
the Kullback--Leibler divergence small compared with \(\log|\mathcal A|\).

\emph{Separation.}  If \(\alpha,\beta\in\mathcal A\) are distinct, then
orthonormality of the Walsh basis gives
\begin{align*}
  \norm{f_\alpha-f_\beta}_2^2
  &=
  a^2\sum_{S\in\mathcal L_m}(\alpha_S-\beta_S)^2 \\
  &=
  4a^2 d_H(\alpha,\beta) \\
  &\geqslant
  4c_2 a^2L_m
  =
  c_2\,\frac{B}{m}.
\end{align*}
Thus the family \(\{f_\alpha:\alpha\in\mathcal A\}\) is \(2\delta\)-separated
in \(L^2(\mu)\), with
\[
  \delta^2=\frac{c_2B}{4m}.
\]
Moreover, relative to the zero function,
\[
  \norm{f_\alpha}_2^2
  =
  a^2L_m
  =
  \frac{B}{4m}.
\]

\emph{Kullback--Leibler control.}  For the lower bound it is enough to
consider the Gaussian noise submodel with variance \(\sigma^2\).  This
submodel satisfies the noise assumption, and any lower bound obtained for
this admissible experiment is also a lower bound for the minimax risk
under the stated model.

In this Gaussian submodel, with common random design distribution \(\mu\),
the Kullback--Leibler divergence from the model with regression function
\(0\) is
\[
  \KL(P_\alpha,P_0)
  =
  \frac{n}{2\sigma^2}\norm{f_\alpha}_2^2
  =
  \frac{nB}{8\sigma^2 m}.
\]
On the other hand,
\[
  \log |\mathcal A|
  \geqslant
  c_1L_m
  \geqslant
  c_1c_\kappa n.
\]
Hence, uniformly over \(\alpha\in\mathcal A\),
\[
  \frac{\KL(P_\alpha,P_0)}{\log |\mathcal A|}
  \leqslant
  \frac{B}{8\sigma^2 c_1c_\kappa}\,\frac{1}{m}.
\]
The right-hand side is at most \(1/4\) as soon as
\(m\geqslant B/(2\sigma^2c_1c_\kappa)\).  Hence, whenever
\(D^*(d,n)\geqslant D_0\), with \(D_0\) depending only on \(B\),
\(\sigma^2\), and \(\kappa\),
\[
  \frac{1}{|\mathcal A|}
  \sum_{\alpha\in\mathcal A}
  \KL(P_\alpha,P_0)
  \leqslant
  \frac14\log |\mathcal A|.
\]
This is condition \eqref{eq:fano-base} in Lemma~\ref{lem:fano}, with the
zero function as reference point.

Applying Lemma~\ref{lem:fano} to the family
\[
  \{0\}\cup\{f_\alpha:\alpha\in\mathcal A\},
\]
we obtain
\[
  \inf_{\wh f}\sup_{f\in\cE_B}R(\wh f,f)
  \geqslant
  \delta^2
  \left(
    1-\frac14-\frac{\log 2}{\log|\mathcal A|}
  \right).
\]
Moreover, by the definition of \(D^*(d,n)\) and the assumption
\(d\geqslant\kappa(D^*(d,n)+1)\),
\[
  n\geqslant N(d,D^*)\geqslant\binom{d}{D^*}
  \geqslant\Bigl(\frac{d}{D^*}\Bigr)^{D^*}
  \geqslant\kappa^{D^*},
\]
so that \(\log|\mathcal A|\geqslant c_1c_\kappa\,\kappa^{D^*}\).  Hence,
after increasing \(D_0\) if necessary, in a way that depends only on
\(\kappa\), the factor in parentheses is bounded from below by \(1/2\)
whenever \(D^*(d,n)\geqslant D_0\).  Using \(\delta^2=c_2B/(4m)\), we get
\[
  \inf_{\wh f}\sup_{f\in\cE_B}R(\wh f,f)
  \geqslant
  c\,\frac{B}{m},
\]
where \(c>0\) is a universal constant.

Finally, since \(m=D^*(d,n)+1\) and \(D^*(d,n)\geqslant D_0\geqslant1\),
we have \(m\leqslant 2D^*(d,n)\).  Consequently,
\[
  \inf_{\wh f}\sup_{f\in\cE_B}R(\wh f,f)
  \geqslant
  c'\,\frac{B}{D^*(d,n)}
\]
for another universal constant \(c'>0\). This proves
the lower bound in Theorem~\ref{thm:ellipsoid}.  Enlarging \(D_0\) once
more if necessary so that \(D_0\geqslant D_1\), where \(D_1\) is the
threshold of Lemma~\ref{lem:balanced-truncation} applied with
\(V=\sigma^2+B/4\), the upper bound holds as well.  Since this choice of
\(V\) is a function of \(B\) and \(\sigma^2\) only, the enlarged
\(D_0\) still depends only on \(B\), \(\sigma^2\), and \(\kappa\).  This
completes the proof of Theorem~\ref{thm:ellipsoid}.

\begin{remark}
The lower-bound construction uses only the first unresolved Fourier layer.
The regular-resolution assumption ensures that this layer is large enough:
by Lemma~\ref{lem:growth}(iii),
\[
  \binom{d}{D^*(d,n)+1}\geqslant c_\kappa n .
\]
Thus the layer immediately beyond the resolution boundary already contains
at least a constant multiple of \(n\) orthogonal Fourier directions.  This
is enough to support a Fano packing, and is the mechanism behind the lower
bound.
\end{remark}

\subsection{Proof of Corollary~\ref{cor:regimes}}

Let
\[
  r=\frac{\log n}{\log d}.
\]
Under the assumptions of the corollary, \(r\to\infty\), \(r=o(d)\), and
\(\log r=o(\log d)\).

We first prove the lower bound on \(D^*(d,n)\).  Let \(0<c<1\) and set
\(D=\lfloor cr\rfloor\).  Since \(D=o(d)\), the binomial coefficients
\(\binom dk\) are increasing for \(0\leqslant k\leqslant D\), for all
sufficiently large \(n\).  Hence
\[
  N(d,D)
  =
  \sum_{k=0}^D\binom dk
  \leqslant
  (D+1)\binom dD
  \leqslant
  (D+1)\left(\frac{ed}{D}\right)^D .
\]
Taking logarithms gives
\[
  \log N(d,D)
  \leqslant
  \log(D+1)+D\log\left(\frac{ed}{D}\right).
\]
Since \(D\asymp r\) and \(\log r=o(\log d)\), we have
\[
  \log D=O(\log r)=o(\log d),
\]
and therefore
\[
  \log\left(\frac{ed}{D}\right)
  =
  (1+o(1))\log d.
\]
Also,
\[
  \log(D+1)=o(D\log d).
\]
Using \(D=\lfloor cr\rfloor=(c+o(1))r\), we obtain
\[
  \log N(d,D)
  \leqslant
  (c+o(1))\,r\log d
  =
  (c+o(1))\log n .
\]
Choosing \(c<1\) fixed, this implies \(N(d,D)\leqslant n\) for all
sufficiently large \(n\).  Hence \(D\leqslant D^*(d,n)\), and therefore
\[
  D^*(d,n)\gtrsim r .
\]

We now prove the upper bound on \(D^*(d,n)\).  Let \(C>1\) and set
\(D=\lfloor Cr\rfloor\).  Again \(D=o(d)\).  Since
\[
  N(d,D)\geqslant \binom dD
  \geqslant
  \left(\frac dD\right)^D,
\]
we get
\[
  \log N(d,D)
  \geqslant
  D\log\left(\frac dD\right)
  =
  D\log d-D\log D .
\]
Because \(D=\lfloor Cr\rfloor=(C+o(1))r\) and
\(\log D=o(\log d)\), this yields
\[
  \log N(d,D)
  \geqslant
  (C+o(1))\,r\log d
  =
  (C+o(1))\log n .
\]
Choosing \(C>1\) fixed, this implies \(N(d,D)>n\) for all sufficiently large
\(n\).  Hence \(D>D^*(d,n)\), and therefore
\[
  D^*(d,n)\lesssim r .
\]
Combining the two bounds gives
\[
  D^*(d,n)\asymp r
  =
  \frac{\log n}{\log d}.
\]
Since \(\log d=(\log n)^b\), we conclude that
\[
  D^*(d,n)\asymp(\log n)^{1-b}.
\]
The regular-window condition follows because \(D^*(d,n)\to\infty\) and
\(D^*(d,n)=o(d)\).  The rate statement follows from
Theorem~\ref{thm:ellipsoid}.

\clearpage
\section{Supplementary Material}
\label{sec:suppl}

\subsection{Proof of Theorem~\ref{thm:known-support}}

The upper bound is a direct
projection argument on the known cube.  The lower bound uses a monotone
packing on the middle layer of \(\Hs\).

\begin{lemma}[Resolvable degree on a logarithmic cube]\label{lem:known-dstar}
Assume that, for some fixed constants \(\varepsilon>0\) and
\(C_0<\infty\),
\[
  (1+\varepsilon)\log_2 n
  \leqslant
  s
  \leqslant
  C_0\log_2 n.
\]
Then, for all sufficiently large \(n\), the pair \((d,n)=(s,n)\) satisfies
the inequalities in Assumption~\ref{ass:regular}, with constants depending
only on \(\varepsilon\) and \(C_0\). Moreover,
\[
  D^*(s,n)\asymp s\asymp \log n.
\]
\end{lemma}

\begin{proof}
Let
\[
  H_2(u)=-u\log_2 u-(1-u)\log_2(1-u)
\]
be the binary entropy function.

We first prove a lower bound on \(D^*(s,n)\).  Choose
\(\eta\in(0,1/2)\) such that
\[
  C_0H_2(\eta)<1.
\]
By the entropy bound for binomial coefficients
(Lemma~\ref{lem:binomial-entropy}),
\[
  \sum_{k=0}^{\lfloor\eta s\rfloor}\binom{s}{k}
  \leqslant
  2^{sH_2(\eta)}.
\]
Since \(s\leqslant C_0\log_2 n\), the right-hand side is bounded by
\[
  n^{C_0H_2(\eta)}=o(n).
\]
Thus, for all sufficiently large \(n\),
\[
  N(s,\lfloor\eta s\rfloor)
  \leqslant n,
\]
and hence
\[
  D^*(s,n)\geqslant \lfloor\eta s\rfloor.
\]
Since \(s\to\infty\), after decreasing \(\eta\) and relabeling the resulting
constant, we obtain
\[
  D^*(s,n)\geqslant \eta s
\]
for all sufficiently large \(n\).

We now prove that \(D^*(s,n)\) remains below a fixed fraction of \(s\),
strictly smaller than \(1/2\).  Since \(\varepsilon>0\), we can choose
\(\rho\in(0,1/2)\) such that
\[
  H_2(\rho)>\frac{1}{1+\varepsilon}.
\]
Since
\[
  s\geqslant(1+\varepsilon)\log_2 n,
\]
we have \(s\to\infty\).  Therefore, by the second bound in
Lemma~\ref{lem:binomial-entropy}, there exists \(c_\rho>0\) such that,
for all sufficiently large \(n\),
\[
  \binom{s}{\lfloor\rho s\rfloor}
  \geqslant
  c_\rho s^{-1/2}2^{sH_2(\rho)}.
\]
On the other hand,
\[
  s\geqslant (1+\varepsilon)\log_2 n
  \qquad\Longrightarrow\qquad
  n\leqslant 2^{s/(1+\varepsilon)}.
\]
Hence
\[
  \frac{\binom{s}{\lfloor\rho s\rfloor}}{n}
  \geqslant
  c_\rho s^{-1/2}
  2^{s\left(H_2(\rho)-1/(1+\varepsilon)\right)}
  \longrightarrow\infty .
\]
Thus, for all sufficiently large \(n\),
\[
  \binom{s}{\lfloor\rho s\rfloor}>n .
\]
Therefore
\[
  N(s,\lfloor\rho s\rfloor)>n,
\]
and by the definition of the resolvable degree,
\[
  D^*(s,n)<\lfloor\rho s\rfloor.
\]
Since \(\rho<1/2\), choose \(\kappa\) such that
\[
  2<\kappa<\frac1\rho .
\]
Since \(D^*(s,n)\) is integer,
\[
  D^*(s,n)+1
  \leqslant
  \lfloor\rho s\rfloor
  \leqslant
  \rho s
\]
for all sufficiently large \(n\).  Therefore
\[
  \kappa\bigl(D^*(s,n)+1\bigr)
  \leqslant
  \kappa\rho s
  <s.
\]
Thus the inequalities in Assumption~\ref{ass:regular} hold on the
\(s\)-dimensional cube, with constants depending only on
\(\varepsilon\) and \(C_0\).

Finally, for all sufficiently large \(n\),
\[
  \eta s
  \leqslant
  D^*(s,n)
  \leqslant
  \rho s,
\]
and hence
\[
  D^*(s,n)\asymp s.
\]
Since
\[
  (1+\varepsilon)\log_2 n
  \leqslant s
  \leqslant
  C_0\log_2 n,
\]
we also have \(s\asymp\log n\).  Therefore
\[
  D^*(s,n)\asymp s\asymp\log n.
\]
\end{proof}

We now prove the upper bound in Theorem~\ref{thm:known-support}.  By
Lemma~\ref{lem:known-dstar}, Assumption~\ref{ass:regular} holds on the
\(s\)-dimensional cube and
\[
  D^*(s,n)\asymp s .
\]
Every \(g\in\cG_{s,B}\) takes values in \([0,1]\), so the moment bound in
Proposition~\ref{prop:projection} holds with \(V=\sigma^2+1\) (unlike the centered
ellipsoid, where \(V=\sigma^2+B/4\) follows from \(\wh f(\emptyset)=0\),
here the bound \(\E g(X)^2\leqslant1\) comes directly from
\(0\leqslant g\leqslant1\)).  Moreover,
\(I^{(2)}(g)\leqslant B\).  Applying the upper-bound argument of Theorem~\ref{thm:ellipsoid} on
\(\Hs\), with \(d=s\) and \(V=\sigma^2+1\), gives
\[
  \sup_{g\in\cG_{s,B}}R(\wh g_{D_\sharp},g)
  \leqslant
  C\,\frac{B}{D^*(s,n)}
  \leqslant
  C'\,\frac{B}{s},
\]
where \(C'\) depends only on \(\varepsilon\) and \(C_0\).  This proves the 
upper bound.

It remains to prove the lower bound.  For \(z\in\Hs\), write $|z|=\sum_{i=1}^s z_i$
for its Hamming weight.  Put \(m=\lfloor s/2\rfloor\), and let
\[
  \Lambda_m=\{z\in\Hs: |z|=m\},
  \qquad
  M_m=|\Lambda_m|=\binom{s}{m}.
\]
By Stirling's formula, there exist universal constants
\(0<c_M<C_M<\infty\) such that, for all \(s\geqslant1\),
\[
  c_M\,\frac{2^s}{\sqrt{s}}
  \leqslant
  M_m
  \leqslant
  C_M\,\frac{2^s}{\sqrt{s}}.
\]
Since \(s\geqslant(1+\varepsilon)\log_2 n\), this gives
\[
  M_m
  \geqslant
  c_M\,\frac{n^{1+\varepsilon}}{\sqrt{s}}.
\]

By the Varshamov--Gilbert bound \citep[Lemma~4.7]{massart2007}, there
exist universal constants \(c_1>0\) and \(c_2>0\), and a subset
\(\Omega\subset\{0,1\}^{\Lambda_m}\), such that
\begin{equation}\label{eq:VG-middle}
  \log |\Omega|\geqslant c_1M_m,
  \qquad
  d_H(\omega,\omega')\geqslant c_2M_m
  \quad
  \text{for all distinct }\omega,\omega'\in\Omega,
\end{equation}
where \(d_H\) denotes the Hamming distance on
\(\{0,1\}^{\Lambda_m}\).

For a scale \(\beta\in(0,1]\), to be fixed below, define
\begin{equation}\label{eq:middle-family}
  g_\omega(z)=
  \begin{cases}
  0, & |z|<m,\\
  \beta\omega_z, & z\in\Lambda_m,\\
  \beta, & |z|>m.
  \end{cases}
\end{equation}

\begin{lemma}[Middle-layer family]\label{lem:middle-family}
For \(\beta\in(0,1]\), the functions \(g_\omega\) are coordinatewise
nondecreasing and take values in \([0,1]\).  Moreover, there exists a universal constant
\(A_0<\infty\) such that
\[
  I^{(2)}(g_\omega)
  \leqslant
  A_0\beta^2\sqrt{s}
  \qquad
  \text{for all }\omega\in\Omega.
\]
\end{lemma}

\begin{proof}
The range constraint is immediate from the definition
\eqref{eq:middle-family}.  We first check monotonicity.  Let
\(z,z'\in\Hs\) be such that \(z\leqslant z'\) coordinatewise.  Then
\(|z|\leqslant |z'|\).  If \(|z|<m\), then
\(g_\omega(z)=0\leqslant g_\omega(z')\).  If \(|z|>m\), then necessarily
\(|z'|>m\), and hence \(g_\omega(z)=g_\omega(z')=\beta\).  It remains to
consider \(|z|=m\).  If \(|z'|=m\), then \(z\leqslant z'\) and
\(|z|=|z'|\) imply \(z=z'\).  If \(|z'|>m\), then
\(g_\omega(z)\in\{0,\beta\}\) and \(g_\omega(z')=\beta\).  Thus
\(g_\omega\) is coordinatewise nondecreasing.

We turn to the Fourier--Sobolev energy bound.  Fix \(i\in[s]\) and
\(z\in\Hs\), and put
\[
  k=\sum_{j\ne i}z_j.
\]
Then \(|z^{i\to0}|=k\) and \(|z^{i\to1}|=k+1\).  The increment
\(\Delta_i g_\omega(z)\) can be nonzero only when these two endpoints meet
the middle layer, that is, only when \(k=m-1\) or \(k=m\).  If \(k=m-1\),
then \(z^{i\to1}\in\Lambda_m\) and
\[
  \Delta_i g_\omega(z)=\beta\omega_{z^{i\to1}}.
\]
If \(k=m\), then \(z^{i\to0}\in\Lambda_m\) and
\[
  \Delta_i g_\omega(z)=\beta(1-\omega_{z^{i\to0}}).
\]
In all other cases, \(\Delta_i g_\omega(z)=0\).  Thus
\[
  |\Delta_i g_\omega(z)|\leqslant\beta
  \qquad
  \text{for all }z\in\Hs.
\]

Let \(X\) be uniformly distributed on \(\Hs\), and set
\[
  K_i=\sum_{j\ne i}X_j.
\]
Then \(K_i\sim\mathrm{Bin}(s-1,1/2)\), and the preceding pointwise
description gives
\[
  (\Delta_i g_\omega(X))^2
  \leqslant
  \beta^2\one_{\{K_i=m-1\}}
  +
  \beta^2\one_{\{K_i=m\}}.
\]
Therefore
\[
  \E[(\Delta_i g_\omega(X))^2]
  \leqslant
  \beta^2
  \frac{\binom{s-1}{m-1}+\binom{s-1}{m}}{2^{s-1}}
  =
  \beta^2\frac{\binom{s}{m}}{2^{s-1}}.
\]
Using the derivative representation \eqref{eq:fourier-sobolev-derivative},
summing over \(i=1,\ldots,s\), and using
\[
  \frac{\binom{s}{m}}{2^s}\leqslant \frac{C_M}{\sqrt{s}},
\]
we obtain
\[
  I^{(2)}(g_\omega)
  =
  \sum_{i=1}^s \E[(\Delta_i g_\omega(X))^2]
  \leqslant
  s\,\beta^2\frac{\binom{s}{m}}{2^{s-1}}
  \leqslant
  A_0\beta^2\sqrt{s}.
\]
\end{proof}

Choose
\begin{equation}\label{eq:beta}
  \beta^2=\frac{B}{A_0\sqrt{s}}.
\end{equation}
Since \(B>0\) is fixed and \(s\to\infty\), we have \(\beta\leqslant1\) for
all sufficiently large \(n\).  Lemma~\ref{lem:middle-family} therefore
gives \(g_\omega\in\cG_{s,B}\) for all \(\omega\in\Omega\).

For distinct \(\omega,\omega'\in\Omega\), the functions \(g_\omega\) and
\(g_{\omega'}\) differ only on \(\Lambda_m\).  Hence
\[
  \norm{g_\omega-g_{\omega'}}_2^2
  =
  \frac{\beta^2 d_H(\omega,\omega')}{2^s}.
\]
Using \eqref{eq:VG-middle}, the lower bound on \(M_m\), and
\eqref{eq:beta}, we get
\[
  \norm{g_\omega-g_{\omega'}}_2^2
  \geqslant
  c_2\beta^2\frac{M_m}{2^s}
  \geqslant
  \frac{c_2c_M}{A_0}\,\frac{B}{s}.
\]
Thus the family \(\{g_\omega:\omega\in\Omega\}\) is separated in
\(L^2(\mu)\) at squared distance of order \(B/s\).  

Choose \(\omega_0\in\Omega\) and set \(\Omega_1=\Omega\setminus\{\omega_0\}\);
for all sufficiently large \(n\), \(|\Omega_1|\geqslant2\).  We verify the
two hypotheses of Lemma~\ref{lem:fano} for the family
\(\{g_{\omega_0}\}\cup\{g_\omega:\omega\in\Omega_1\}\), with metric
\(\rho(g,g')=\norm{g-g'}_2\).

\emph{Separation.}  The lower bound just obtained gives, for all distinct
\(\omega,\omega'\in\Omega_1\),
\[
  \norm{g_\omega-g_{\omega'}}_2\geqslant2\delta,
  \qquad
  \delta^2=\frac{c_2c_M}{4A_0}\,\frac{B}{s}.
\]

\emph{Kullback--Leibler control.}  Restrict to the Gaussian noise
submodel with variance \(\sigma^2\).  Using
\(d_H(\omega,\omega')\leqslant M_m\), the upper bound on \(M_m\), and
\eqref{eq:beta},
\[
  \KL(P_\omega,P_{\omega'})
  =
  \frac{n}{2\sigma^2}\norm{g_\omega-g_{\omega'}}_2^2
  \leqslant
  \frac{C_M}{2A_0\sigma^2}\,\frac{nB}{s}
\]
for all \(\omega,\omega'\in\Omega\).  On the other hand, by
\eqref{eq:VG-middle} and the lower bound on \(M_m\),
\[
  \log|\Omega_1|\sim\log|\Omega|\geqslant c_1M_m
  \geqslant c_1c_M\,\frac{n^{1+\varepsilon}}{\sqrt{s}}.
\]
Combining the last two displays,
\[
 \frac{\sup_{\omega,\omega'\in\Omega,\,\omega\ne\omega'}\KL(P_\omega,P_{\omega'})}{\log|\Omega_1|}
  \leqslant
  \frac{C_M}{2A_0\sigma^2c_1c_M}\,\frac{B}{n^\varepsilon\sqrt{s}}
  \to0,
\]
so that, for all sufficiently large \(n\), this ratio is at most
\(1/4\).  Thus the sufficient condition \eqref{eq:fano-pair} holds with
\(\alpha=1/4\), \(M=|\Omega_1|\), and index set \(\Theta=\Omega\).

Applying Lemma~\ref{lem:fano} to the family
\[
  \{g_{\omega_0}\}\cup\{g_\omega:\omega\in\Omega_1\}
\]
with the metric \(\rho(g,g')=\norm{g-g'}_2\), we obtain
\[
  \inf_{\wh g}\sup_{g\in\cG_{s,B}}R(\wh g,g)
  \geqslant
  \delta^2
  \left(
    1-\frac14-\frac{\log 2}{\log|\Omega_1|}
  \right).
\]
Since \(\log|\Omega_1|\to\infty\), the factor in parentheses is bounded
from below by a positive constant for all sufficiently large \(n\).
Therefore
\[
  \inf_{\wh g}\sup_{g\in\cG_{s,B}}R(\wh g,g)
  \geqslant
  c\,\frac{B}{s}.
\]
This completes the proof of Theorem~\ref{thm:known-support}.

\subsection{Proof of Theorem~\ref{thm:intrinsic-monotone}}

The lower bound follows
from Theorem~\ref{thm:known-support}, since the unknown-support class
\(\cM_{d,B,s}\) contains the subclass of functions depending on any fixed
set of \(s\) coordinates. Indeed, fix the first \(s\) coordinates and identify each
\(g\in\cG_{s,B}\) with its cylindrical extension
\[
  f(x)=g(x_1,\ldots,x_s),
  \qquad x\in\Hd .
\]
Then \(f\in\cM_{d,B,s}\) and \(I^{(2)}(f)=I^{(2)}(g)\). The remaining coordinates are independent of the response throughout
this submodel.  We make the reduction to the known-support experiment
explicit.  Write \(X=(Z,U)\), where \(Z\in\Hs\) collects the first
\(s\) coordinates and \(U\in\{0,1\}^{d-s}\) the remaining ones, so
that \(Z\) and \(U\) are independent and uniform, and
\(Y=g(Z)+\varepsilon\).  First, a statistician observing the
known-support sample \((Z_j,Y_j)_{j=1}^n\) may draw auxiliary variables
\(U_1,\ldots,U_n\), i.i.d.\ uniform on \(\{0,1\}^{d-s}\) and
independent of everything else; the pairs
\(((Z_j,U_j),Y_j)_{j=1}^n\) then have exactly the distribution of the
\(d\)-dimensional sample.  Second, given any estimator \(\wh f\) in
the \(d\)-dimensional experiment, define
\[
  \wh g(z)
  =
  \E\bigl[\wh f(z,U)\bigm|(Z_j,Y_j)_{j=1}^n\bigr],
  \qquad z\in\Hs,
\]
where the conditional expectation averages both over \(U\), uniform on
\(\{0,1\}^{d-s}\) and independent of the rest, and over the auxiliary
variables \(U_1,\ldots,U_n\); in particular, \(\wh g\) is an
estimator in the known-support experiment.  Since \(f(z,u)=g(z)\) does
not depend on \(u\), Jensen's inequality for conditional expectations
gives, pointwise in \(z\),
\[
  \bigl(\wh g(z)-g(z)\bigr)^2
  \leqslant
  \E\Bigl[\bigl(\wh f(z,U)-f(z,U)\bigr)^2\Bigm|(Z_j,Y_j)_{j=1}^n\Bigr],
\]
and taking expectations and integrating over \(z\in\Hs\) yields
\[
  \E\norm{\wh g-g}_{L^2(\Hs)}^2
  \leqslant
  \E\norm{\wh f-f}_{L^2(\Hd)}^2 .
\]
Thus every estimator in the \(d\)-dimensional experiment induces an
estimator in the known-support experiment with no larger risk, so the
minimax risk over the cylindrical submodel is bounded from below by the
minimax risk of the known-support experiment on \(\Hs\), and the lower
bound follows from Theorem~\ref{thm:known-support}.

We therefore focus on the upper bound.  The following proposition gives the
required support-selection version of the projection estimator.

\begin{proposition}[Projection with unknown support]\label{prop:unknown-support-projection}
Let \(n\geqslant2\), \(1\leqslant s\leqslant d\), and \(0\leqslant D\leqslant s\). There
exists an estimator \(\wh f_{s,D}\) such that, for every
\(f\in\cM_{d,B,s}\),
\[
  \E\norm{\wh f_{s,D}-f}_2^2
  \leqslant
  C\left\{
  \frac{B}{D+1}
  +
  \frac{(\sigma^2+1)\bigl(N(s,D)+\log\binom ds\bigr)}{n}
  \right\},
\]
where \(C<\infty\) is a universal constant.
\end{proposition}

\begin{proof}
For each candidate support \(J\subset[d]\) with \(|J|=s\), compute on the
first half of the sample the empirical Fourier projection of degree \(D\) on
\[
  \mathcal V(J,D)=\operatorname{span}\{\chi_S:S\subseteq J,
  |S|\leqslant D\},
\]
and clip it to \([0,1]\).  Denote the resulting candidate by \(\wh f_J\).

Let \(J_f=\operatorname{supp}(f)\).  Since \(f\in\cM_{d,B,s}\), one has
\(|J_f|\leqslant s\).  Hence \(J_f\) is contained in at least one candidate
support \(J\subset[d]\) with \(|J|=s\).  For any such \(J\), the function
\(f\) depends only on the coordinates in \(J\), and its Fourier--Walsh
coefficients vanish outside subsets of \(J\).  Therefore the estimator \(\wh f_J\) reproduces the known-support
projection, and the bias--variance decomposition in the proof of
Proposition~\ref{prop:projection} applies verbatim with the index set
\(\{S\subseteq J:|S|\leqslant D\}\), of cardinality \(N(s,D)\), in
place of \(\{S:|S|\leqslant D\}\): the squared bias equals
\(\sum_{S\subseteq J,\,|S|>D}\wh f(S)^2\), which is at most
\(B/(4(D+1))\) by \eqref{eq:tail-bound}, since
\(I^{(2)}(f)\leqslant B\); each retained empirical coefficient,
computed from the first half of the sample, of size
\(\lfloor n/2\rfloor\geqslant n/3\), has variance at most
\((\sigma^2+1)/\lfloor n/2\rfloor\); and clipping to \([0,1]\) can
only decrease the error.  We obtain
\[
  \E\norm{\wh f_J-f}_2^2
  \leqslant
  C\left(
  \frac{B}{D+1}
  +
  \frac{(\sigma^2+1)N(s,D)}{n}
  \right).
\]

The second half of the sample is used only to select among the
\(\binom ds\) clipped candidates by empirical squared loss.  Let
\(\widehat J\) be a minimizer of the empirical squared loss over the second
half of the sample among the candidates \(\{\wh f_J:|J|=s\}\), and set
\[
  \wh f_{s,D}=\wh f_{\widehat J}.
\]
Conditionally on the first half of the sample, the family
\(\{\wh f_J:|J|=s\}\) is fixed and has cardinality at most
\(M=\binom ds\).  The second half of the sample, of size
\(m=n-\lfloor n/2\rfloor\geqslant n/2\), is independent of this family.
Since the candidates and \(f\) take values in \([0,1]\) and the noise is
sub-Gaussian, Lemma~\ref{lem:holdout}, applied conditionally on the
first half of the sample, yields, after taking expectation with respect
to the first half and using
\(\E\inf_{|J|=s}\leqslant\inf_{|J|=s}\E\) together with
\(m\geqslant n/2\),
\[
  \E\norm{\wh f_{s,D}-f}_2^2
  \leqslant
  2\inf_{|J|=s}\E\norm{\wh f_J-f}_2^2
  +
  C\,\frac{(\sigma^2+1)\log\bigl(e\binom ds\bigr)}{n}.
\]
Finally,
\[
  \log\Bigl(e\binom ds\Bigr)
  =
  1+\log\binom ds,
\]
and the additive constant \(1\) is absorbed into the term
\((\sigma^2+1)N(s,D)/n\) of the stated bound, since
\(N(s,D) \geqslant \binom s 0 \geqslant 1\).  This absorption also covers the degenerate case
\(s=d\), where \(\binom ds=1\), the logarithm \(\log\binom ds\)
vanishes, and the selection step involves a single candidate.  Combining
the resulting inequality with the bound for any candidate support
containing \(J_f\) proves the result.
\end{proof}

\begin{remark}
Oracle inequalities of this type are classical; see, for instance,
\citet{wegkamp2003} for closely related statements.  We include a short 
self-contained proof in Section~\ref{ssec:holdout}, adapted to the 
sub-Gaussian noise framework of the present paper.
\end{remark}

We now complete the proof of Theorem~\ref{thm:intrinsic-monotone}.  Choose
\(\eta>0\) small enough and set \(D=\lfloor\eta s\rfloor\).  More precisely, decrease \(\eta\) if necessary so that
\[
  C_0H_2(\eta)<1-\gamma
\]
for some \(\gamma\in(0,1)\), where \(H_2\) is the binary entropy function.  As in
the proof of Lemma~\ref{lem:known-dstar}, the entropy bound gives
\[
  N(s,D)
  \leqslant
  2^{sH_2(\eta)}
  \leqslant
  n^{C_0H_2(\eta)}
  \leqslant
  n^{1-\gamma}.
\]
Since \(s\asymp\log n\), it follows that
\[
  \frac{N(s,D)}{n}
  \leqslant
  n^{-\gamma}
  =
  o\left(\frac1s\right).
\]
Applying Proposition~\ref{prop:unknown-support-projection}, we obtain, for
a universal constant \(C'<\infty\),
\[
  R_n^*(\cM_{d,B,s})
  \leqslant
  C'\left\{
  \frac{B}{D+1}
  +
  \frac{(\sigma^2+1)N(s,D)}{n}
  +
  \frac{(\sigma^2+1)\log\binom ds}{n}
  \right\}.
\]
Here \(D+1\asymp s\), so the first term is of order \(B/s\); and by the
previous display, for all sufficiently large \(n\),
\[
  \frac{(\sigma^2+1)N(s,D)}{n}
  \leqslant
  \frac{B}{s}.
\]
Note that the threshold may depend on \(B\), \(\sigma^2\), \(\varepsilon\), and
\(C_0\).  Therefore
\[
  R_n^*(\cM_{d,B,s})
  \leqslant
  C\left(
  \frac{B}{s}
  +
  \frac{\log\binom ds}{n}
  \right),
\]
where \(C<\infty\) depends only on \(\sigma^2\), \(\varepsilon\), and
\(C_0\).  Together with the lower bound obtained from
Theorem~\ref{thm:known-support}, this proves the first assertion of the
theorem. The final assertion follows immediately from the condition
\[
  \log\binom ds
  \leqslant
  A_1\frac{nB}{s}.
\]
The constants in the resulting equivalence \(R_n^*(\cM_{d,B,s})\asymp
B/s\) depend on \(\sigma^2\), \(\varepsilon\), \(C_0\), and additionally
on \(A_1\).  This completes the proof of
Theorem~\ref{thm:intrinsic-monotone}.

\subsection{Entropy bounds for binomial coefficients}

\begin{lemma}[Entropy bounds for binomial coefficients]\label{lem:binomial-entropy}
Let
\[
  H_2(u)=-u\log_2 u-(1-u)\log_2(1-u),
  \qquad u\in(0,1).
\]

For every integer \(s\geqslant1\) and every
\(\eta\in(0,1/2]\),
\[
  \sum_{k=0}^{\lfloor\eta s\rfloor}\binom{s}{k}
  \leqslant
  2^{sH_2(\eta)}.
\]

Moreover, for every fixed \(\rho\in(0,1)\), there exist constants
\(c_\rho>0\) and \(s_\rho\geqslant1\) such that, for every
\(s\geqslant s_\rho\),
\[
  \binom{s}{\lfloor\rho s\rfloor}
  \geqslant
  c_\rho s^{-1/2}2^{sH_2(\rho)}.
\]
\end{lemma}

\begin{proof}
We first prove the upper bound.  Fix \(\eta\in(0,1/2]\).  By the binomial
theorem,
\[
  1
  =
  \sum_{k=0}^s
  \binom{s}{k}\eta^k(1-\eta)^{s-k}.
\]
Since \(\eta/(1-\eta)\leqslant1\), the function
\[
  t\longmapsto \eta^t(1-\eta)^{s-t}
\]
is nonincreasing.  Hence, for every
\(0\leqslant k\leqslant\lfloor\eta s\rfloor\),
\[
  \eta^k(1-\eta)^{s-k}
  \geqslant
  \eta^{\eta s}(1-\eta)^{(1-\eta)s}.
\]
Therefore
\[
  1
  \geqslant
  \eta^{\eta s}(1-\eta)^{(1-\eta)s}
  \sum_{k=0}^{\lfloor\eta s\rfloor}\binom{s}{k},
\]
and thus
\[
  \sum_{k=0}^{\lfloor\eta s\rfloor}\binom{s}{k}
  \leqslant
  \eta^{-\eta s}(1-\eta)^{-(1-\eta)s}
  =
  2^{sH_2(\eta)}.
\]

We now prove the lower bound.  Let
\[
  m=\lfloor\rho s\rfloor,
  \qquad
  p_s=\frac{m}{s}.
\]
Then \(p_s\to\rho\).  In particular, both \(m\) and \(s-m\) tend to
infinity.

By Stirling's formula, there exists a universal constant \(c_S>0\)
such that, for every integer \(k\geqslant1\),
\[
  c_S\sqrt k\left(\frac ke\right)^k
  \leqslant
  k!
  \leqslant
  c_S^{-1}\sqrt k\left(\frac ke\right)^k.
\]
Applying these bounds to \(s!\), \(m!\), and \((s-m)!\), and noting
that the powers of \(e\) cancel, we obtain
\[
  \binom{s}{m}
  \geqslant
  c_S^3
  \sqrt{\frac{s}{m(s-m)}}
  \frac{s^s}{m^m(s-m)^{s-m}} .
\]
Since \(m=sp_s\),
\[
  \binom{s}{m}
  \geqslant
  \frac{c_S^3}{\sqrt{s\,p_s(1-p_s)}}
  2^{sH_2(p_s)}.
\]
For all sufficiently large \(s\), one has \(p_s\in(0,1)\), and hence
\[
  \frac{1}{\sqrt{p_s(1-p_s)}}\geqslant2.
\]
Moreover,
\[
  |p_s-\rho|\leqslant\frac1s.
\]
Since \(H_2\) is continuously differentiable in a neighborhood of
\(\rho\), there exists \(C_\rho<\infty\) such that
\[
  |H_2(p_s)-H_2(\rho)|
  \leqslant
  \frac{C_\rho}{s}
\]
for all sufficiently large \(s\).  Therefore
\[
  2^{sH_2(p_s)}
  \geqslant
  2^{-C_\rho}2^{sH_2(\rho)}.
\]
Combining the preceding inequalities yields
\[
  \binom{s}{\lfloor\rho s\rfloor}
  \geqslant
  c_\rho s^{-1/2}2^{sH_2(\rho)}
\]
for all sufficiently large \(s\).
\end{proof}

\subsection{A hold-out inequality for finite classes}\label{ssec:holdout}

\begin{lemma}[Finite hold-out selection]\label{lem:holdout}
Let \((X_j,Y_j)_{j=1}^m\) be an i.i.d.\ validation sample satisfying
\[
  Y_j=f(X_j)+\varepsilon_j,
\]
where \(0\leqslant f\leqslant 1\), and where the noise variables are
centered, independent of the design, and sub-Gaussian with parameter
\(\sigma^2\), as in Section~\ref{sec:model}.  Let
\(\mathcal H=\{h_1,\ldots,h_M\}\) be a finite family of fixed functions
such that \(0\leqslant h\leqslant 1\) for all \(h\in\mathcal H\).

For \(h\in\mathcal H\), define
\[
  \ell(h)=\E\big[(f(X)-h(X))^2\big],
  \qquad
  \widehat \ell(h)=\frac1m\sum_{j=1}^m (Y_j-h(X_j))^2,
\]
where \(X\) is distributed as the design variables.  Let
\[
  \widehat h\in\arg\min_{h\in\mathcal H}\widehat \ell(h).
\]
Then there exists a universal constant \(C<\infty\) such that
\[
  \E \ell(\widehat h)
  \leqslant
  2\inf_{h\in\mathcal H}\ell(h)
  +
  C\,\frac{(\sigma^2+1)\log(eM)}{m}.
\]
\end{lemma}

\begin{proof}
Let
\[
  h^\star\in\arg\min_{h\in\mathcal H}\ell(h),
  \qquad
  \ell^\star=\ell(h^\star).
\]
For \(h\in\mathcal H\), define the excess validation loss
\[
  Z_h(x,y)
  =
  (y-h(x))^2-(y-h^\star(x))^2 ,
\]
and write \(Z_h=Z_h(X,Y)\) for the corresponding random variable, where
\((X,Y)\) is a generic observation. Writing \(Y=f(X)+\varepsilon\) and \(\psi_h=h-h^\star\) (a function on
the design space, so that \(\psi_h(X)\) denotes its value at \(X\)), a direct
expansion gives
\[
  Z_h
  =
  (f-h)^2-(f-h^\star)^2
  -2\varepsilon\,\psi_h(X).
\]
Since \(\E[\varepsilon\mid X]=0\),
\[
  \mu_h
  =
  \E Z_h
  =
  \ell(h)-\ell^\star .
\]
Note that \(|\psi_h|\leqslant1\) and that, since
\(0\leqslant f,h\leqslant 1\), one has \(\ell(h)\leqslant1\) for every
\(h\in\mathcal H\); in particular, \(\mu_h\leqslant1\).

\medskip
\emph{Step 1.}
Fix \(h\in\mathcal H\) such that \(\ell(h)\geqslant2\ell^\star\) and
\(\ell(h)>0\).  Then
\[
  \mu_h=\ell(h)-\ell^\star\geqslant \frac12 \ell(h)>0,
  \qquad\text{so that}\qquad
  0<\mu_h\leqslant1 .
\]
Moreover, since \(\psi_h=(h-f)+(f-h^\star)\),
\[
  \E \psi_h(X)^2
  \leqslant
  2\ell(h)+2\ell^\star
  \leqslant
  3\ell(h)
  \leqslant
  6\mu_h .
\]

We claim that \(Z_h-\mu_h\) satisfies the Bernstein moment condition:
for every integer \(q\geqslant2\),
\begin{equation}\label{eq:bernstein-condition}
  \E|Z_h-\mu_h|^q
  \leqslant
  \frac{q!}{2}\,v\,b^{q-2},
  \qquad
  v=K(\sigma^2+1)\mu_h,
  \qquad
  b=K(\sigma^2+1),
\end{equation}
with \(K=192\).  To see this, write \(A=(f-h)^2-(f-h^\star)^2\), a
function on the design space with \(A=A(X)\) understood, so that
\(Z_h-\mu_h=(A-\mu_h)-2\varepsilon\psi_h(X)\), and, factoring the
difference of squares,
\[
  |A|
  =
  |\psi_h|\,|2f-h-h^\star|
  \leqslant
  2|\psi_h| .
\]
Fix an integer \(q\geqslant2\).  Since \(|\psi_h|\leqslant1\), one has
\(|\psi_h|^q\leqslant\psi_h^2\) and \(|A|^q\leqslant2^q\psi_h^2\), while
\(0<\mu_h\leqslant1\) gives \(\mu_h^q\leqslant\mu_h\).  Hence, by
convexity of \(x\mapsto x^q\),
\[
  \E|A-\mu_h|^q
  \leqslant
  2^{q-1}\bigl(\E|A|^q+\mu_h^q\bigr)
  \leqslant
  2^{q-1}\bigl(2^q\cdot6\mu_h+\mu_h\bigr)
  \leqslant
  7\cdot2^{2q-1}\mu_h .
\]
On the other hand, integrating the Chernoff tail bound
\({\mathbb P}\{|\varepsilon|>t\}\leqslant2e^{-t^2/(2\sigma^2)}\) yields
\(\E|\varepsilon|^q\leqslant q!\,(\sqrt2\sigma)^q\), so that, by
independence of \(\varepsilon\) and \(X\),
\[
  \E\big|2\varepsilon\psi_h(X)\big|^q
  =
  2^q\,\E|\varepsilon|^q\,\E|\psi_h(X)|^q
  \leqslant
  2^q\,q!\,(\sqrt2\sigma)^q\cdot6\mu_h .
\]
Combining the two displays by convexity again, and using
\(q!/2\geqslant1\) for the first term,
\begin{align*}
  \E|Z_h-\mu_h|^q
  &\leqslant
  2^{q-1}\,\E|A-\mu_h|^q
  +
  2^{q-1}\,\E\big|2\varepsilon\psi_h(X)\big|^q\\
  &\leqslant
  7\cdot2^{3q-2}\mu_h
  +
  3\,q!\,(4\sqrt2\sigma)^q\mu_h\\
  &=
  112\cdot8^{q-2}\mu_h
  +
  \frac{q!}{2}\cdot192\sigma^2(4\sqrt2\sigma)^{q-2}\mu_h\\
  &\leqslant
  \frac{q!}{2}\,
  \bigl(112+192\sigma^2\bigr)
  \max\bigl(8,4\sqrt2\sigma\bigr)^{q-2}\mu_h .
\end{align*}
Since \(112+192\sigma^2\leqslant192(\sigma^2+1)\) and
\(\max(8,4\sqrt2\sigma)\leqslant8(\sigma^2+1)\leqslant192(\sigma^2+1)\),
this proves \eqref{eq:bernstein-condition} with \(K=192\).

We now invoke Bernstein's inequality in the one-sided form of
\citet[Theorem~2.10]{boucheron2013}: if \(W_1,\ldots,W_m\) are
independent real-valued random variables such that
\(\sum_{j=1}^m\E[W_j^2]\leqslant V\) and
\(\sum_{j=1}^m\E[(W_j)_+^q]\leqslant\frac{q!}{2}Vc^{q-2}\) for every
integer \(q\geqslant3\), then, for every \(t>0\),
\[
  {\mathbb P}\left\{
  \sum_{j=1}^m \bigl(W_j-\E W_j\bigr)\geqslant \sqrt{2Vt}+ct
  \right\}
  \leqslant
  e^{-t}.
\]
We apply this to the centered variables
\(W_j=\mu_h-Z_h(X_j,Y_j)\).  By \eqref{eq:bernstein-condition} and
\((x)_+\leqslant|x|\), the hypotheses hold with \(V=mv\) and \(c=b\).
Take
\[
  t_0=\frac{m\mu_h^2}{2(v+b\mu_h)},
  \qquad\text{so that}\qquad
  \sqrt{2mv\,t_0}+b\,t_0
  =
  m\mu_h
  \left(
  \sqrt{1-x}+\frac{x}{2}
  \right)
  \leqslant
  m\mu_h,
\]
where \(x=b\mu_h/(v+b\mu_h)\in(0,1)\) and we used
\(\sqrt{1-x}\leqslant1-x/2\).  Therefore
\[
  {\mathbb P}\left\{
    \frac1m\sum_{j=1}^m Z_h(X_j,Y_j)\leqslant0
  \right\}
  =
  {\mathbb P}\left\{
    \sum_{j=1}^m W_j\geqslant m\mu_h
  \right\}
  \leqslant
  e^{-t_0}
  =
  \exp\left(
    -\frac{m\mu_h^2}{2(v+b\mu_h)}
  \right).
\]
Since \(v+b\mu_h\leqslant 2K(\sigma^2+1)\mu_h\), and since
\(\mu_h\geqslant \ell(h)/2\), we conclude that
\begin{equation}\label{eq:fixed-h-bound}
  {\mathbb P}\left\{
    \frac1m\sum_{j=1}^m Z_h(X_j,Y_j)\leqslant0
  \right\}
  \leqslant
  \exp\left(
    -c\,m\,\frac{\ell(h)}{\sigma^2+1}
  \right)
\end{equation}
for a universal constant \(c>0\).

\medskip
\emph{Step 2.}
Fix \(r>0\).  If
\[
  \ell(\widehat h)>2\ell^\star+r,
\]
then \(\widehat h\) belongs to the set
\(\{h\in\mathcal H:\ell(h)>2\ell^\star+r\}\), and, by empirical
minimality, \(\widehat \ell(\widehat h)\leqslant \widehat
\ell(h^\star)\).  Since
\[
  \widehat \ell(h)-\widehat \ell(h^\star)
  =
  \frac1m\sum_{j=1}^m Z_h(X_j,Y_j),
\]
it follows that
\[
  \{\ell(\widehat h)>2\ell^\star+r\}
  \subseteq
  \bigcup_{\substack{h\in\mathcal H\\ \ell(h)>2\ell^\star+r}}
  \left\{
  \frac1m\sum_{j=1}^m Z_h(X_j,Y_j)\leqslant0
  \right\}.
\]
Every \(h\) in this union satisfies \(\ell(h)>2\ell^\star+r\); in
particular, \(\ell(h)\geqslant2\ell^\star\) and
\(\ell(h)\geqslant r>0\), so that the bound \eqref{eq:fixed-h-bound}
applies.  A union bound over \(\mathcal H\) therefore gives
\[
  {\mathbb P}\{\ell(\widehat h)>2\ell^\star+r\}
  \leqslant
  M\exp\left(
    -c\,m\,\frac{r}{\sigma^2+1}
  \right).
\]

\medskip
\emph{Step 3.}
For \(u\geqslant0\), take
\[
  r
  =
  K_0\frac{(\sigma^2+1)(\log(eM)+u)}{m}
\]
with \(K_0=2/c\).  Then
\[
  M\exp\left(-c\,m\,\frac{r}{\sigma^2+1}\right)
  =
  M(eM)^{-2}e^{-2u}
  =
  \frac{1}{e^2M}\,e^{-2u}
  \leqslant
  e^{-u},
\]
so that
\[
  {\mathbb P}\left\{
    \ell(\widehat h)
    >
    2\ell^\star
    +
    K_0\frac{(\sigma^2+1)(\log(eM)+u)}{m}
  \right\}
  \leqslant
  e^{-u},
  \qquad u\geqslant0.
\]
Consequently, the random variable
\[
  T
  =
  \frac{m}{K_0(\sigma^2+1)}
  \left(
    \ell(\widehat h)-2\ell^\star-K_0\frac{(\sigma^2+1)\log(eM)}{m}
  \right)_+
\]
satisfies \({\mathbb P}\{T>u\}\leqslant e^{-u}\) for all \(u\geqslant0\); the
positive part makes \(T\) nonnegative, so that
\(\E T=\int_0^\infty{\mathbb P}\{T>u\}\,\mathrm du\leqslant1\).  Rearranging,
and using \(\log(eM)\geqslant1\),
\[
  \E \ell(\widehat h)
  \leqslant
  2\ell^\star
  +
  K_0\frac{(\sigma^2+1)(\log(eM)+1)}{m}
  \leqslant
  2\ell^\star
  +
  2K_0\frac{(\sigma^2+1)\log(eM)}{m}.
\]
Since \(\ell^\star=\inf_{h\in\mathcal H}\ell(h)\), the result follows.
\end{proof}


\bibliographystyle{apalike}
\bibliography{references}

\end{document}